\documentclass[12pt]{article}
\usepackage{amsmath, amsthm, amssymb}
\usepackage{amscd}
\usepackage[all,cmtip]{xy}
\usepackage{ulem}
\usepackage[boxsize=3pt]{ytableau}
\usepackage[left=1in,top=1in,right=1in,bottom=1in,nohead]{geometry}
\usepackage{graphicx}
\usepackage{tikz}
\usepackage{stmaryrd}
\usetikzlibrary{matrix,calc}
\usepackage{mathrsfs}
\usepackage{cite}

\usepackage[color]{showkeys}

\definecolor{refkey}{rgb}{1,1,1}
\definecolor{labelkey}{rgb}{1,1,1}
\definecolor{cite}{rgb}{0.9451,0.2706,0.4941}
\definecolor{ruri}{rgb}{0.0078,0.4022,0.8010}

\usepackage{hyperref}

\newtheorem{thm}{Theorem}[section]
\newtheorem{lem}[thm]{Lemma}
\newtheorem{defn}[thm]{Definition}
\newtheorem{cor}[thm]{Corollary}

\newtheorem{prop}[thm]{Proposition}
\newtheorem{obs}[thm]{Observation}
\newtheorem{rmk}[thm]{Remark}
\newtheorem{ex}[thm]{Example}
\newtheorem{quest}[thm]{Question}

\newcommand{\bi}{\begin{itemize}}
\newcommand{\ei}{\end{itemize}}
\newcommand{\ben}{\begin{enumerate}}
\newcommand{\een}{\end{enumerate}}
\newcommand{\be}{\begin{equation}}
\newcommand{\ee}{\end{equation}}
\newcommand{\bea}{\begin{eqnarray}}
\newcommand{\eea}{\end{eqnarray}}
\newcommand{\bal}{\begin{align}}
\newcommand{\eal}{\end{align}}
\newcommand{\ba}{\begin{array}}
\newcommand{\ea}{\end{array}}
\newcommand{\nn}{\nonumber}
\newcommand{\sm}[1]{\scalebox{.65}{$#1$}}

\newcommand{\catname}[1]{{\normalfont\textbf{#1}}}
\newcommand{\mop}[1]{{\normalfont\text{#1}}}
\newcommand{\Fin}{\mathcal{F}}
\newcommand{\QFin}{\mathbb{Q}\Fin}
\newcommand{\Sym}{\mathcal{S}}

\newcommand{\Bij}{\mathcal{S}}

\newcommand{\Vect}{\mathcal{V}}
\newcommand{\FRep}{\Vect^{\Fin}}
\newcommand{\UP}{\Vect^{\Fin}}
\newcommand{\BRep}{\Vect^{\Bij}}

\newcommand{\tensorpower}{\scalebox{1.2}{$\boldmath{\boldmath{\otimes}}$\hspace{-.8pt}}}

\newcommand{\fg}{\Vect^{\Fin}_{\scalebox{.6}{$f \hspace{-1pt}g$}}}
\newcommand{\Sp}{S\hspace{-1pt}p}
\newcommand{\fn}[1]{\scalebox{.7}{$\fbox{\texttt{#1}}$}}

\newcommand{\Hom}{\mop{Hom}}

\newcommand{\Ext}{\mop{Ext}}
\newcommand{\Aut}{\mop{Aut}}

\newcommand{\Tr}{\mop{Tr}}
\newcommand{\im}{\mop{Im}}
\newcommand{\coker}{\mop{coker}}
\newcommand{\Iso}{\mop{Iso}}
\newcommand{\Lan}{\mop{Lan}}

\newcommand{\Res}{\mop{Res}}
\newcommand{\Ind}{\mop{Ind}}

\newcommand{\sk}{\mop{sk}}

\makeatletter
\newenvironment{amatrix}{%
  \left\langle
  \env@matrix
}{%
  \endmatrix
  \right\rangle
}

\newcommand{\amatrixarrow}[3][.5ex]{%
  \def\amatrixarrow@shift{#1}%
  \def\amatrixarrow@left{#2}%
  \def\amatrixarrow@right{#3}%
  \collect@body\amatrixarrow@next
}
\newcommand*{\amatrixarrow@next}[1]{%
  \mathpalette{\amatrixarrow@}{#1}%
}
\newcommand*{\amatrixarrow@}[2]{%
  \sbox0{$\m@th#1\begin{amatrix}#2\end{amatrix}$}%
  \sbox2{$\m@th#1\begin{matrix}#2\end{matrix}$}%
  \sbox4{$\m@th
    \amatrixarrow@style{#1}%
    \amatrixarrow@left
  $}%
  \sbox6{$\m@th
    \amatrixarrow@style{#1}%
    \amatrixarrow@right
  $}%
  \sbox8{$\m@th\amatrixarrow@style{#1}\text{\kern\amatrixarrow@shift}$}%
  \dimen0=.5\dimexpr\wd0-\wd2\relax
  \vtop{%
    \hbox{%
      \kern.5\dimexpr\wd4-\dimen0\relax
      \copy0 %
    }%
    \kern\wd8 %
    \nointerlineskip
    \hbox to \dimexpr\wd0+.5\wd4+.5\wd6-\dimen0\relax{%
      \copy4 %
      \hfill
      \sbox2{$\m@th
        \amatrixarrow@style{#1}%
        {}\xrightarrow{}{}%
      $}
      $\m@th
        \amatrixarrow@style{#1}%
        \xrightarrow{%
          \kern\dimexpr\wd0-.5\wd4-.5\wd6-\dimen0-\wd2\relax
        }%
      $%
      \hfill
      \copy6 %
    }%
  }%
}
\newcommand*{\amatrixarrow@style}[1]{%
  \ifx#1\displaystyle
    \scriptstyle
  \else\ifx#1\textstyle
    \scriptstyle
  \else
    \scriptscriptstyle
  \fi\fi
}

\newcommand{\bmatrixarrow}[3][.5ex]{%
  \def\bmatrixarrow@shift{#1}%
  \def\bmatrixarrow@left{#2}%
  \def\bmatrixarrow@right{#3}%
  \collect@body\bmatrixarrow@next
}
\newcommand*{\bmatrixarrow@next}[1]{%
  \mathpalette{\bmatrixarrow@}{#1}%
}
\newcommand*{\bmatrixarrow@}[2]{%
  \sbox0{$\m@th#1\begin{bmatrix}#2\end{bmatrix}$}%
  \sbox2{$\m@th#1\begin{matrix}#2\end{matrix}$}%
  \sbox4{$\m@th
    \bmatrixarrow@style{#1}%
    \bmatrixarrow@left
  $}%
  \sbox6{$\m@th
    \bmatrixarrow@style{#1}%
    \bmatrixarrow@right
  $}%
  \sbox8{$\m@th\bmatrixarrow@style{#1}\text{\kern\bmatrixarrow@shift}$}%
  \dimen0=.5\dimexpr\wd0-\wd2\relax
  \vtop{%
    \hbox{%
      \kern.5\dimexpr\wd4-\dimen0\relax
      \copy0 %
    }%
    \kern\wd8 %
    \nointerlineskip
    \hbox to \dimexpr\wd0+.5\wd4+.5\wd6-\dimen0\relax{%
      \copy4 %
      \hfill
      \sbox2{$\m@th
        \bmatrixarrow@style{#1}%
        {}\xrightarrow{}{}%
      $}
      $\m@th
        \bmatrixarrow@style{#1}%
        \xrightarrow{%
          \kern\dimexpr\wd0-.5\wd4-.5\wd6-\dimen0-\wd2\relax
        }%
      $%
      \hfill
      \copy6 %
    }%
  }%
}
\newcommand*{\bmatrixarrow@style}[1]{%
  \ifx#1\displaystyle
    \scriptstyle
  \else\ifx#1\textstyle
    \scriptstyle
  \else
    \scriptscriptstyle
  \fi\fi
}

\DeclareRobustCommand{\stirling}{\genfrac\{\}{0pt}{}}

\title{Uniformly Presented Vector Spaces}
\author{ John D. Wiltshire-Gordon\thanks{I am pleased to thank: Benson Farb and Thomas Church for introducing me to  representation stability and inspiring me with their original paper \cite{ChurchFarbRTHS}; my advisor David Speyer for providing knowledge, encouragement, and wisdom; also, Mitya Boyarchenko, Jordan Ellenberg, Gene Kopp, Vipul Naik, Andrew Snowden, Julian Rosen, and Steven Sam for valuable conversations.  I humbly thank the NSF for generous support through a Graduate Research Fellowship.}\; \; \; \; \texttt{johnwg@umich.edu}}
\begin{document}
\maketitle
\begin{abstract}
Gaussian elimination answers any question about a finitely presented vector space.  However, a ``uniform family'' of such presentations---given as generic relations among an unspecified number of generators---is susceptible to elimination only once the number of generators is fixed.  We develop a theory of ``uniformly presented vector spaces'' to compute with these uniform families, introducing a formalism of finitely generated functors from the category of finite sets $\Fin$ to the category of finite dimensional $\mathbb{Q}$-vector spaces $\Vect$.  We show that these representations have finite length and polynomial dimension away from the empty set, and produce finite leftward resolutions by manageable functors.

\end{abstract}
\section{Introduction} \label{intro}
\subsection{A first example}
Let $U=\langle x_1, x_2, \ldots, x_n \rangle$ be an $n$-dimensional vector space over $\mathbb{Q}$, and let $V$ be the following quotient of $U \otimes U$:
\be \label{origw}
V= \Big \langle x_i \otimes x_j \Big \rangle \; \bigg / \; \Big \langle x_i \otimes x_i + x_j \otimes x_j + x_k \otimes x_k - 3\left(x_i \otimes x_j - x_i \otimes x_k + x_j \otimes x_k \right) \Big \rangle,
\ee
where the indexing variables $i$, $j$, and $k$ are (not-necessarily distinct) elements of the indexing set  $\{1, 2, \ldots , n \}$.  This vector space $V$ is a first example of a ``uniformly presented'' vector space.  Here ``uniformly'' is short for ``uniformly in $n$,'' since the definition of $V$ depends on $n$ in a consistent way. \\ \\
\noindent
The essential characteristic is that the indexing variables are allowed to range freely over a finite set---there are no mentions of $x_{i+1}$, for example, or any other expressions that assume extra structure on the set $\{1, 2, \ldots , n \}$. \\ \\
The explicit presentation for $V$ allows us to compute its dimension for any particular $n$ by row reduction: build a matrix $M$ whose $n^3$ rows come from all possible substitutions
\be \nn
(i,j,k) \in \{1, 2, \ldots, n \} \times \{1, 2, \ldots, n \} \times \{1, 2, \ldots, n \};
\ee
the dimension of $V$ will be $n^2 - \mop{rank} \left( M \right)$.  \\ \\
\noindent
Experimentation with various $n$ suggests a pattern for the dimension of $V$.  Indeed, an elementary but \textit{ad hoc} calculation gives
\be  \nn
\dim V = n.
\ee
Permuting the set $\{1, 2, \ldots, n \}$ induces an action of the symmetric group $\mathcal{S}_n$ on the vector space $V$.  A similar calculation gives a full report on the character $\chi^V$ associated to $V$:
\be \label{easychar}
\chi^V(\sigma) = \# \{ \mbox{fixed points of $\sigma$} \}.
\ee
The fundamental goal of this paper is to produce a general algorithm taking a uniform presentation like (\ref{origw}) and computing a simple formula for its character like (\ref{easychar}).
\subsection{The Formalism of $\Fin$-Representations}
Example (\ref{origw}) has a great deal of exploitable structure.  In particular, any function relating indexing sets
\be \nn
f: \{1, 2, \ldots, p \} \longrightarrow \{1, 2, \ldots, q \}
\ee
induces a linear map
\be \nn
f_* : V[p] \longrightarrow V[q],
\ee
where $V[n]$ indicates the vector space associated with a particular choice of $n$.  In other words, $V$ is a \textit{functor} from the category of finite sets to the category of finite dimensional vector spaces over the rational numbers:
\be \nn
V : \Fin \longrightarrow \Vect.
\ee
Not all such functors deserve to be called ``uniformly presented vector spaces;'' roughly speaking, only those which can be written as a quotient
\be \nn
\langle \mbox{finite list of generic generators} \rangle \; \bigg / \; \langle \mbox{finite list of generic relations} \rangle.
\ee
Later we will give an intrinsic characterization of those functors $V : \Fin \longrightarrow \Vect$ which can be written in this form. \\ \\
\noindent
Setting notation for initial segments of the natural numbers
\be \nn
[p] = \{ 1, 2, 3, \ldots, p \},
\ee
we reconcile the notation $V[p]$ and the category-theoretic tradition of applying a functor to an object by writing them adjacently. \\ \\
\noindent
Classical representation theory of a group $G$ can be placed in a parallel framework, if so desired: $G$ can be considered a category with one object, and a finite dimensional representation of $G$ defines an object in the category $\Vect^G$ of functors from $G$ to vector spaces.  From this point of view, it is natural to call a functor from $\Fin$ to $\Vect$ a ``representation of the category $\Fin$,'' or an ``$\Fin$-representation.'' 
\subsection{Summary of results}
The graded dimension of an $\Fin$-representation $V$ is a sequence of natural numbers:
\be \nn
\dim V[0], \;\; \dim V[1], \;\; \dim V[2], \; \; \dim V[3], \;\; \dim V[4], \;\; \ldots, \;\; \dim V[n], \;\; \ldots
\ee
If this sequence is bounded above by a polynomial in $n$, we prove:
\begin{itemize}
\item it actually coincides with a polynomial when $n \geq 1$;
\item in that range, $\chi^V(\sigma)$ is a polynomial in the fixed-point counts of the powers of $\sigma$;
\item $V$ can be written in the form of a uniformly presented vector space.
\end{itemize}
These results follow from structural results about the category of uniformly presented vector spaces, considered as $\Fin$-representations.  For example:
\begin{itemize}
\item Uniformly presented vector spaces have finite length as $\Fin$-representations;
\item Uniformly presented vector spaces form an abelian subcategory of $\Vect^{\Fin}$;
\item Uniformly presented vector spaces admit a finite, algorithmically computable resolution by easier representations.
\end{itemize}
This last bullet (once we make it precise) is the main theorem.  We give a preliminary version as Theorem \ref{mainA} and a refined version as Theorem \ref{mainB}.
\begin{rmk}
Bullets 1, 4, and 5 can be deduced from the classical Dold-Kan theorem \cite{doldclassical}.  However, such a proof forgets the actions of the symmetric groups!  We prefer to build a more general theory, keeping these group actions at the forefront.
\end{rmk}
\section{Prior work}
\subsection{The category of $\mathcal{FI}$-modules}
Church, Ellenberg, and Farb introduced $\mathcal{FI}$-modules in \cite{FIMod} in order to clarify the ``representation stability'' phenomena Church and Farb had observed in \cite{ChurchFarbRTHS}.  An $\mathcal{FI}$-module is a functor from the category of finite sets with injections to some category of modules.  In \cite{FINoether}, working with Nagpal, they obtain strong results regarding the eventual behavior of finitely generated $\mathcal{FI}$-modules over any Noetherian ring $R$.  The case $R = \mathbb{Z}$ is of utmost importance, but the case $R=\mathbb{Q}$ is already subtle.\\ \\
\noindent
Of course, any $\Fin$-representation restricts to an $\mathcal{FI}$-module over $\mathbb{Q}$.  However, it turns out that the representation theory of $\Fin$ is much more rigid than that of $\mathcal{FI}$ and so this restriction map has little utility.  Conversely, finding an $\mathcal{FI}$-module is much easier than finding an $\Fin$-representation; after all, there are fewer induced maps to produce!  For example, $SL_n(\mathbb{Z})$ is naturally $\mathcal{FI}$-functorial in the variable $n$: $\mathcal{FI}$ morphisms induce inclusions of the corresponding groups.  It is clear, however, that this functor has no chance of extending to $\Fin$.  Thus, the group homology of $SL_n(\mathbb{Z})$ is an $\mathcal{FI}$-module but not an $\Fin$-module.  The $\mathcal{FI}$ structure of $SL_n({\mathbb{Z}})$ is behind the scenes in Putman's work \cite{homcong}, and is explicitly invoked in a subsequent paper with Church and Farb in \cite{unstable}. \\ \\
\noindent
The paper \cite{FIMod} also introduces\footnote{The calculations giving semisimplicity and classifying irreducibles are essentially present in \cite{FISharp85}, which uses the language of semigroups instead of categories.} the source category $\mathcal{FI}^{\#}$, the category of finite sets with partial injections.  The representation theory of this category is much more comparable with that of $\Fin$.  The are marked differences, however.  It is easier\footnote{although there is no actual restriction functor from $\Vect^{\mathcal{FI}^{\#}}$ to $\Vect^{\Fin}$. } and more natural to look for $\Fin$-modules in nature.  On the other hand, once you have an $\mathcal{FI}^{\#}$-module, the theory is cleaner since the category of $\mathcal{FI}^{\#}$-modules is semisimple and the category of $\Fin$-modules is not. \\ \\
\noindent
The moral of this story: ``take care with source categories!''  Even among categories built out of symmetric groups---as $\mathcal{FI}$, $\Fin$, and $\mathcal{FI}^{\#}$ are---the stories diverge in complexity and character.
\subsection{Cosimplicial vector spaces}
A classical construction known as the ``Dold-Kan correspondence'' \cite{doldclassical} gives a complete description of the representation theory of $\Delta$, the category of non-empty finite ordinals with order-preserving maps.  Such representations are known as ``cosimplicial vector spaces.''  The Dold-Kan correspondence provides an explicit equivalence of categories between $\Vect^{\Delta}$ and the category of cochain complexes in $\Vect$ supported in non-negative degree.  \\ \\ \noindent
Restricting an $\Fin$-representation $V$ to a $\Delta$-representation forgets the action of the symmetric groups, but it does lend more flexibility.  For example, arguments relying only on $\Delta$-structure extend to any abelian target category.  Also, since restriction to $\Delta$ preserves the dimension sequence, dimension computations are often simpler using only $\Delta$-structure. \\ \\ \noindent
Recent work of Lack and Street \cite{comboeq} generalizes the Dold-Kan construction to include the case $\mathcal{FI^{\#}}$, and other similar cases.  Their work represents a connection to pure category theory.
\subsection{Functor homology}
The popularity of the category of \textit{pointed} finite sets $\Gamma$ as a source category goes back to Segal's infinite loopspace machine \cite{SegalClassical}.  Motivated by this work, Pirashvili studied functors from $\Gamma$ to the category of (possibly non-abelian) groups\cite{PiraDoldKan}.  Subsequent work showed the power and flexibility of homological algebra in a category of functors.  For example, Pirashvili gives a construction of Hochschild and cyclic homology in \cite{HochFunctorHomology}.  \\ \\ \noindent
Working in this vein, Antosz and Betley study the homological algebra of $\Gamma$-representations in \cite{HomAlgGamma}, but over a finite field $\mathbb{F}_q$.  The story seems harder there than in characteristic zero.

\subsection{Twisted commutative algebras}
Sam and Snowden introduced twisted commutative algebras in \cite{GLEquivariantInfinite}, placing the category of $\mathcal{FI}$-modules in a much larger context.  Although there seems to be no direct interaction between twisted commutative algebras and the $\Fin$-representations of this paper, the results do bear striking similarities.  One fundamental difference between the two theories is the role of the monoidal structure on $\mathcal{FI}$ provided by disjoint union.  The theory of $\Fin$-modules makes no reference to this structure, while twisted commutative algebras use it extensively.  \\ \\ \noindent
Twisted commutative algebras can also handle many important sequences of groups besides the symmetric groups $\Bij_n$; see \cite{StabilityReps}.

\section{Summary of notation}
Apart from the final batch of bullet points, the reader is encouraged not to skip this section due to the ubiquity of unusual notation in the rest of the paper.
\subsubsection{Category theory}
In what follows, let $\mathcal{C}$ and $\mathcal{D}$ be categories, and let $X \overset{f}{\longrightarrow} Y \overset{g}{\longrightarrow} Z$ be composable arrows in $\mathcal{C}$.
\begin{itemize}
\item The composite map $X \longrightarrow Z$ is written $fg$.  In other words, maps act on the right.
\item The category of functors from $\mathcal{D}$ to $\mathcal{C}$ is written exponentially $\mathcal{C}^{\mathcal{D}}$.
\item The $\Hom$ functor in any category is written directly with the category name, in keeping with the modern style:
\be \nn
\mathcal{C}(X, Y) = \Hom_{\mathcal{C}}(X, Y).
\ee
\item Triple $\Hom$ notation $\mathcal{C}(X, Y, Z)$ refers to ``all maps from $X$ to $Z$ that factor through maps to $Y$;'' when $\mathcal{C}$ is additive, we need ``linear combinations of maps from $X$ to $Z$ that factor through $Y$.''  See Definition \ref{triplehom}.
\item Adjacent elements are to be composed whenever they are composable.  For example,
\be \nn
f\mathcal{C}(Y, Z) \subseteq \mathcal{C}(X, Z)
\ee
denotes the maps from $X$ to $Z$ that factor through $f$.  We also have
\be \nn
\mathcal{C}(X, Y) \mathcal{C}(Y, Z) = \mathcal{C}(X, Y,Z )\subseteq \mathcal{C}(X, Z).
\ee 
\item Given a functor $F : \mathcal{C} \longrightarrow \mathcal{D}$ and an object $X \in \mathcal{C}$, we write $FX$ for the image of the object under $F$.  Similarly, for an arrow $f \in \mathcal{C}$, we write $Ff$ for the induced map under $F$.  Occasionally we write $f_*$ (or $f^*$ for contravariant functors).
\end{itemize}
\noindent
Here are some source categories we will need.
\begin{itemize}
\item $\Fin$ denotes the full subcategory of $\catname{Set}$ spanned by the finite sets $[k] = \{ 1, 2, 3, \ldots, k \}$.  An arrow in $\Fin$ is written as a function in one-line notation surrounded by a frame.  This example cements our conventions:
\be \nn
[7] \xrightarrow{\; \fn{3322225} \; } [5] \xrightarrow{\; \fn{78911} \; } [9] \; \; \; \; = \; \; \; \; [7] \xrightarrow{\; \fn{9988881} \; } [9].
\ee
\item $\Fin_{\leq k}$ denotes the full subcategory of $\Fin$ spanned by the sets with cardinality at most $k$.
\item $\Bij = \Iso(\Fin)$ denotes the wide subcategory of $\Fin$ consisting of the isomorphisms.  In other words, $\Bij = \bigsqcup_k \Sym_k$, the disjoint union of the symmetric groups.
\item $\Sym_k = \Aut([k])$ denotes the one-object subcategory of $\Fin$ consisting of the finite set $[k]$ and its automorphisms.  In other words, $\Sym_k$ is the symmetric group consisting of permutations of the numbers $\{ 1, 2, 3, \ldots, k \}$.
\item $\Sym_{\leq k} = \Iso(\Fin_{\leq k})$ denotes the disjoint union of symmetric groups $\Sym_0 \sqcup \Sym_1 \sqcup \cdots \sqcup \Sym_k$.
\end{itemize}
Here are some categories of representations we will need.
\begin{itemize}
\item $\Vect$ denotes the category of finite dimensional vector spaces over $\mathbb{Q}$.
\item $\Vect^{\Fin}$ denotes the category of all finite dimensional\footnote{Here, finite dimensional means each $[k]$ is sent to a finite dimensional space.} representations of $\Fin$.
\item $\fg$ denotes the category of finitely generated $\Fin$-representations, which is to say, the full subcategory of $\FRep$ consisting of the finitely-generated functors.  See Section \ref{goodcategory}.
\item $\BRep$ denotes the category of all finite dimensional representations of $\Bij$.\footnote{This category is also known as the category of linear species.}
\end{itemize}
We make frequent use of restriction functors and left Kan extensions relating these categories.
\begin{itemize}
\item Given a functor $K : \mathcal{C} \longrightarrow \mathcal{D}$ (which is usually an inclusion), we have a natural notion of restriction whereby a $\mathcal{D}$-representation becomes a $\mathcal{C}$-representation.  The restriction functor---written $\Res_{\mathcal{C}}^{\mathcal{D}}$---has a left adjoint, the left Kan extension functor $\Lan_{\mathcal{C}}^{\mathcal{D}}$.  We have a natural isomorphism
\be \nn
\Vect^{\mathcal{C}}\left(W, \; \Res_{\mathcal{C}}^{\mathcal{D}} V \right) \simeq \Vect^{\mathcal{D}}\left(\Lan_{\mathcal{C}}^{\mathcal{D}} W, \;  V \right).
\ee
The unit and counit of this adjunction are written $\eta_{\mathcal{C}}^{\mathcal{D}}$ and $\epsilon_{\mathcal{C}}^{\mathcal{D}}$ respectively.  These functors generalize induction and restriction of group representations; see \cite{CategoricalHomotopyTheory}.
\end{itemize}
The following notations are more onerous and we leave their complete definitions elsewhere.  In what follows, $k \in \mathbb{N}$ and $\lambda$ is a partition of $k$.  This list can be safely skipped in a first reading.
\begin{itemize}
\item $\mathbb{Q}\Fin$ is the category of matrices over the free $\mathbb{Q}$-enrichment of $\Fin$; see Section \ref{qfindef}.  Expressions like
\be \nn
\begin{bmatrixarrow}[0ex]{[1] \oplus [3]}{[1] \oplus [4]}
    \fn{1} & 2 \cdot \fn{4}-\fn{3}-\fn{1} \\
    \fn{111} & \fn{323}+\fn{421}
 \end{bmatrixarrow}
\ee
refer to arrows in $\mathbb{Q}\Fin$.
\item $\langle - \rangle$ Angle bracket notation for imreps is Definition \ref{bracketnotation}.
\item $\tensorpower^k$, $\tensorpower^{<k}$, and $\tensorpower^{\leq k} $ are Definition \ref{tpower}.
\item $c_{\lambda}$ refers to a Young symmetrizer, or any other choice of minimal idempotent of $\mathbb{Q} \Bij_k$ generating an irreducible representation of $\Bij_k$ corresponding to the partition $\lambda$.
\item $Sp_{\lambda}$ refers to the irreducible $\Bij_k$-representation generated by $c_{\lambda} \in \mathbb{Q} \Bij_k$.
\item $\mathbb{P}_{\lambda}$ refers to an ``isotype projective.''  See Definition \ref{isotypeprojective}.
\item $\tau_k$ and $\varepsilon_k$ are Definition \ref{taueps}.  $S^k$, $\Lambda^k$, and $\Theta^k$ are Definition \ref{usefulschurs}.
\item The adjoint pair of functors $\mathbb{S} \dashv \mathbb{R}$ are shorter names for the functors $\Lan_{\Bij}^{\Fin} \dashv \Res_{\Bij}^{\Fin}$.  See Definition \ref{SchurRS}.
\item $\sk_P$, $\sk_{\leq k}$, and $\sk_{<k}$ refer to skeleton functors.  See Definition \ref{skel}.
\item The representations $D_k$ and $C_{\lambda}$ are Definitions \ref{definedk} and \ref{defineclambda} respectively.
\item $d_k$ refers to the map in the Koszul complex $\Lambda^{\bullet}$ induced by $\partial_k$, see Definitions  \ref{defpartial} and  \ref{Koszul}.
\item $ev$ refers to the evaluation map, i.e., the counit of the tensor-hom adjunction.  See Definition \ref{tensorhom}.
\item $H_0$ denotes the zeroth homology functor; $h$ is the map sending a vector to its homology class; see Definition \ref{zerothhomology}.  
\item $\stirling{n}{k}$ denotes a Stirling number of the second kind.  See \cite{EnumCombo1}.
\end{itemize}

\section{The representation theory of $\Fin$}
\subsection{The category $\QFin$} \label{qfindef}
Inspired by the construction of the group algebra $\mathbb{Q}G$ of a group $G$, we develop a ``$\mathbb{Q}$-linearized'' version of the category $\Fin$.  Objects of $\QFin$ are formal direct sums of objects of $\Fin$.  Since $\oplus$ is a biproduct, we need only specify the homomorphisms between two singleton sums: $\QFin(X, Y)$ is all $\mathbb{Q}$-linear combinations of $\Fin$ maps from the finite set $X$ to the finite set $Y$.  Composition is assumed to be bilinear so that $\QFin$ is $\Vect$-enriched.  Concretely, morphisms of $\QFin$ are matrices of linear combinations of $\Fin$ maps, and composition is given by matrix multiplication.

\subsubsection{Representable and im-representable functors}
The representable functors of $\QFin$ are a rich source of functors from $\Fin$ to $\Vect$.  We introduce a convenient notation. 
\begin{defn}[Yoneda bracket notation] \label{bracketnotation}
Given objects $X, Y \in \QFin$, define
\be \nn
\langle X \rangle = \QFin(X,-)
\ee
to be the covariant functor represented by $X$.
Given a map $f: X \longrightarrow Y$, define
\be \nn
\langle f \rangle = \im\left( \QFin(Y,-) \overset{f^*}{\longrightarrow} \QFin(X,-) \right),
\ee
to be the image of the natural transformation $f$ induces contravariantly.
\end{defn}
\noindent
The functors $\langle X \rangle$ and $\langle f \rangle$ have domain $\QFin$, but it also makes sense to evaluate them on objects and morphisms in $\Fin$, which sits as a subcategory.  Note that the two uses of $\langle - \rangle$ are compatible in the sense that
\be \nn
\langle X \rangle \simeq \langle 1_X \rangle.
\ee
Any functor isomorphic to one of the form $\langle f \rangle$ is called an ``im-representable'' functor or an ``imrep'' for short.  The im-representable functors are exactly those functors which can be written as the image of a map between representable functors, by the Yoneda lemma.
\begin{ex} \label{oneline}
Given $k \in \mathbb{N}$, the representable functor $\langle \; [k] \; \rangle$ has an explicit description:
\be \nn
\langle \; [k] \; \rangle  X \simeq \left( \mathbb{Q} X \right)^{\otimes k}.
\ee
In other words, $\langle \; [k] \; \rangle$ sends a finite set $X$ to the $k^{th}$ tensor power of the free vector space on $X$.  Let us make this isomorphism explicit.  Given a pure tensor $x_1 \otimes x_2 \otimes \cdots \otimes x_k \in \tensorpower^k (\mathbb{Q} X)$ considering each $x_i \in X$ as a basis vector of $\mathbb{Q}X$, we associate the function
\bea \nn
[k] & \longrightarrow & X \\ \nn
i & \longmapsto & x_i
\eea
reading the pure tensors as if they were expressing a function in one-line notation.
\end{ex}
\noindent
\begin{defn} \label{tpower}
Define tensor power functors from $\Fin$ to $\Vect$
\bea \nn
\tensorpower^{k\phantom{<}} & = & \langle \; [k] \; \rangle \\ \nn
\tensorpower^{<k} & = & \langle \; [0] \oplus [1] \oplus [2] \oplus \cdots \oplus [k-1] \; \rangle \\ \nn
\tensorpower^{\leq k} & = & \langle \; [0] \oplus [1] \oplus [2] \oplus \cdots \oplus [k] \; \rangle.
\eea
\end{defn}
\noindent
The Yoneda lemma provides an isomorphism:
\be \label{yoneda}
\Vect^{\Fin}\left( \tensorpower^k , \; V \right) \simeq V[k].
\ee
\subsubsection{Lists of generic spanning vectors give imreps and vice-versa}
The bracket notation for imreps is inspired by examples like
\be \nn
V=\langle x_i \otimes x_j \rangle \; \bigg / \; \big \langle x_i \otimes x_i + x_j \otimes x_j + x_k \otimes x_k - 3\left(x_i \otimes x_j - x_i \otimes x_k + x_j \otimes x_k \right) \big \rangle,
\ee
from Section \ref{intro}.  Indeed, $V$ is isomorphic to a quotient of imreps
\be \nn
\begin{amatrixarrow}[0ex]{[2]}{[2]}
   \; \fn{12} \;
 \end{amatrixarrow} \; \Bigg / \;
 \begin{amatrixarrow}[0ex]{[2]}{[3]}
    \; \fn{11} + \fn{22} + \fn{33} -3 \cdot \fn{12} + 3 \cdot \fn{13} -3 \cdot \fn{23}  \;
 \end{amatrixarrow} 
\ee
via the one-line-notation-trick from Example \ref{oneline}.  (We are writing, for example, $\fn{14}$ for the function sending $1 \mapsto 1$ and $2 \mapsto 4$).  As another example, consider the important\footnote{It is $H^1$ of the moduli space of genus 1 curves over $\mathbb{R}$ with $n$ marked points.} uniformly presented vector space (taken from \cite{RealModuli})
\be \label{h1mg}
\langle w_{ijkl} \rangle \, \Big / \, \langle w_{ijlk}+w_{jklm}+w_{klmi}+w_{lmij}+w_{mijk} \rangle
\ee
where the symbol $w_{ijkl}$ is assumed to be antisymmetric in $ijkl$, and $i$, $j$, $k$, and $l$ are assumed to be distinct elements of the set $\{1, 2, 3, \ldots , n \}$.  Once again, we have an isomorphism to a quotient of imreps:
\be \nn
\begin{amatrixarrow}[0ex]{[4]}{[4]}
    \fn{1234}
 \end{amatrixarrow} \, \Bigg / \, \begin{amatrixarrow}[0ex]{[4]}{[4] \oplus [4] \oplus [3] \oplus [5]}
    \fn{1234} + \fn{2134} \;\;\; & \fn{1234}+\fn{2341} \;\;\; & \fn{1123} \;\;\; & \fn{1234}+\fn{2345}+\fn{3451}+\fn{4512}+\fn{5123} 
 \end{amatrixarrow}
\ee
The first two generators in the denominator force the antisymmetry condition\footnote{They force a transposition and a $4$-cycle to act by $-1$, which is enough}, and the third (combined with the first two) imposes distinctness.  It is generally true that convenient presentations like (\ref{h1mg}) can be converted to slightly less wieldy quotients of imreps.  The hope (realized in the proof of Theorem \ref{mainB}) is that imreps might be amenable to systematic or even algorithmic approaches.

\subsubsection{Finitely generated representations}
A finitely generated $R$-module is one which is a quotient of a free module $R^n$ for some $n$; similarly, a finitely generated $\Fin$-representation is a quotient of a representable functor $\langle Y \rangle$ for some $Y$.
\begin{defn}
A functor $V : \Fin \longrightarrow \Vect$
is called \textbf{finitely generated} if there exists an object $Y \in \QFin$ with a surjective map:
\be \nn
\langle Y \rangle \longrightarrow V \longrightarrow 0.
\ee
\end{defn}
\noindent
To produce actual vectors generating $V$, write $Y$ as a sum of finite sets and look at the image of the various universal vectors\footnote{The universal vectors are the identity maps on the finite sets that direct sum to $Y$.} supplied by the Yoneda lemma (\ref{yoneda}).  These special vectors are analogous to the standard basis vectors of $R^n$.\\ \\
\noindent
It is clear from this definition that any imrep is finitely generated; similarly, any imrep is a subobject of some functor of the form $\langle X \rangle$.  More can be said.
\begin{obs} \label{rowcolops}
Given a pair of composable arrows in $\QFin$
\be \nn
X \overset{f}{\longrightarrow} Y \overset{g}{\longrightarrow} Z,
\ee
the imrep $\langle fg \rangle$ is naturally a quotient of $\langle g \rangle$ and a subobject of $\langle f \rangle$.
\end{obs}
\noindent
Practically speaking, Observation \ref{rowcolops} allows us to perform row and column operations on imreps.  One available style of proof uses a $\QFin$ equation of the form $f = p q f r s$.   We know $\langle f \rangle = \langle pqfrs \rangle$ is a subquotient of $\langle qfr \rangle$, and $\langle qfr \rangle$ is a subquotient of $\langle f \rangle$, so $\langle f \rangle \simeq \langle qfr\rangle$.
\begin{ex} \label{sumexample}
Let $X$ be an object of $\QFin$, and suppose we have an endomorphism  $\pi: X \longrightarrow X$ satisfying $\pi \pi = \pi$.  The $\QFin$ compositions

\be \nn
 \hspace{.98in}\begin{bmatrixarrow}[0ex]{X \oplus X}{X}
    \pi \\
    1_X - \pi
 \end{bmatrixarrow} \cdot
 \begin{bmatrixarrow}[0ex]{X}{X}
    1_X
 \end{bmatrixarrow} \cdot
 \begin{bmatrixarrow}[0ex]{X}{X \oplus X}
    \pi & 1_X-\pi
 \end{bmatrixarrow}=\begin{bmatrixarrow}[0ex]{X \oplus X}{X \oplus X}
    \pi & 0 \\
    0 & 1_X - \pi
 \end{bmatrixarrow}
\ee
\vspace{.15in}
\be \nn
\begin{bmatrixarrow}[0ex]{X}{X \oplus X}
    1_X & 1_X
 \end{bmatrixarrow} \cdot
\begin{bmatrixarrow}[0ex]{X \oplus X}{X \oplus X}
    \pi & 0 \\
    0 & 1_X - \pi
 \end{bmatrixarrow} \cdot
\begin{bmatrixarrow}[0ex]{X \oplus X}{X}
    1_X \\
    1_X
 \end{bmatrixarrow}=\begin{bmatrixarrow}[0ex]{X}{X}
    1_X
 \end{bmatrixarrow}
\ee
show that $\begin{bmatrixarrow}[0ex]{X}{X}
    1_X
 \end{bmatrixarrow}$ and \scalebox{.8}{$\begin{bmatrixarrow}[0ex]{X \oplus X}{X \oplus X}
    \pi & 0 \\
    0 & 1_X - \pi
 \end{bmatrixarrow}$} are the same matrix up to row and column operations.  \\ \\ \noindent It follows from Observation \ref{rowcolops} that their imreps are mutual subquotients and so
\be \nn
\begin{amatrixarrow}[0ex]{X}{X}
    1_X
 \end{amatrixarrow} \simeq \begin{amatrixarrow}[0ex]{X \oplus X}{X \oplus X}
    \pi & 0 \\
    0 & 1_X - \pi
 \end{amatrixarrow}.
\ee
Further, the previous computation makes the isomorphism explicit.
\end{ex}
\subsubsection{Definition of uniformly presented vector spaces}
We are in a position to give a precise definition of a uniformly presented vector space.
\begin{defn}
A \textbf{uniformly presented vector space} is an $\Fin$-representation $V$ which is given as a quotient of imreps 
\be \nn
V = \langle f \rangle \; \big / \; \langle fg \rangle
\ee
for some pair of composable morphisms $f, g \in \QFin$.
\end{defn}
\subsubsection{Projective objects}
The representable functors $\langle X \rangle$ are projective as objects of the category of representations $\Vect^{\Fin}$: by the Yoneda lemma, the hom functor $\Vect^{\Fin}( \langle X \rangle, \; -)$ is isomorphic to evaluation at $X$, and thus is exact.  Further, given an idempotent $\pi \in \QFin(X, X)$, the functor $\langle \pi \rangle$ is evidently projective as well, since Example \ref{sumexample} gives
\be \nn
\langle \pi \rangle \oplus \langle 1_X - \pi \rangle \cong \langle X \rangle,
\ee
which writes $\langle \pi \rangle$ as a direct summand of a projective object.
\\ \\
\noindent
The representation theory of symmetric groups supplies a wealth of such projections $\pi$.  Fix a finite set $X$, and let $\Bij_X$ denote the group of permutations of $X$.  The group algebra $\mathbb{Q}\Bij_X$ is generated as a left $\mathbb{Q}\Bij_X$-module by minimal idempotents.  Every minimal idempotent $\pi$ gives rise to a functor $\langle \pi \rangle$ since we have a natural inclusion
\be \nn
\pi \in \mathbb{Q}\Bij_X \subseteq \QFin(X,X).
\ee
\begin{defn} \label{isotypeprojective}
Let $\lambda$ be a partition $k$, and $c_{\lambda} \in \mathbb{Q}\Bij_k$ be the Young symmetrizer (or any other minimal idempotent corresponding to $\lambda$).  The \textbf{isotype projective} $\mathbb{P}_{\lambda}$ is defined to be the imrep
\be \nn
\mathbb{P}_{\lambda} = \langle c_{\lambda} \rangle.
\ee
\end{defn}
\noindent
Any two minimal idempotents associated to $\lambda$ are conjugate in the group algebra, so an argument similar to the one given in Example \ref{rowcolops} gives a (non-canonical) isomorphism between any two such functors.  This argument mostly justifies the notational independence of a choice of minimal idempotent.\\ \\
\noindent
It is worth mentioning that the trivial and sign representations, being one dimensional, do not suffer from this ambiguity: their endomorphism algebras contain a unique minimal idempotent.  We develop notation to handle these two common cases.
\begin{defn} \label{taueps}
The  \textbf{$k^{th}$ trivial idempotent} $\tau_k  \in  \QFin([k],[k])$ is
\be \nn
\tau_k  =  \frac{1}{k!} \sum_{\sigma \in \Bij_k} \sigma.
\ee
Similarly, the  \textbf{$k^{th}$ alternating idempotent} $\varepsilon_k  \in  \QFin([k],[k])$ is
\be \nn
\varepsilon_k  =  \frac{1}{k!} \sum_{\sigma \in \Bij_k} \mathop{sign} (\sigma) \cdot \sigma.
\ee
\end{defn}
\begin{defn} \label{usefulschurs}
In keeping with Definition \ref{tpower}, define the functors
\bea \nn
S^k & = & \langle \tau_k \rangle \\ \nn
\Lambda^k & = & \langle \varepsilon_k \rangle \\ \nn
\Theta^k & = & \langle 1_k - \varepsilon_k \rangle
\eea
We note that all three of these functors are naturally summands of $\tensorpower^k$ and that
\be \nn
\tensorpower^k = \Lambda^k \oplus \Theta^k.
\ee
\end{defn}
\subsection{The universal property of isotype projectives}
We provide a universal property of $\mathbb{P}_{\lambda}$ and give an explicit construction in terms of classical representation theory.  The following result provides a refinement of the earlier observation that $\Vect^{\Fin} (\tensorpower^k, V) \simeq V[k]$.
\begin{prop} \label{univiso}
Let $\lambda$ be a partition of $k$.  For any $\Fin$-representation $V$, the isotype projective $\mathbb{P}_{\lambda}$ satisfies
\be \nn
\Vect^{\Fin}(\mathbb{P}_{\lambda},\, V) \; \simeq \; \Vect^{\Bij_k}(\Sp_{\lambda} , \, V[k])
\ee
where $\Sp_{\lambda}$ denotes the irreducible representation of $\Bij_k$ corresponding to $\lambda$.
\end{prop}
\noindent
In other words, $\dim \Vect^{\Fin}(\mathbb{P}_{\lambda}, V)$ computes the multiplicity of the Specht module $\Sp_{\lambda}$ in the $\Bij_k$-representation $V[k]$.  Further, we see that a map from $S^k$ (resp. $\Lambda^k$) to an $\Fin$-representation $V$ is the same as a vector in the trivial (resp. alternating) isotypic component of $V[k]$.\\ \\
\noindent
The following proposition is analogous to our explicit description of the functor $\tensorpower^k$ as the free functor followed by the usual tensor power, as seen in Example \ref{oneline}.
\begin{prop} \label{schurfree}
Given $\lambda$ a partition of $k \in \mathbb{N}$ and a finite set $X$, $\mathbb{P}_{\lambda}X$ has the following explicit description:
\be \nn
\mathbb{P}_{\lambda}X \simeq \mathbb{S}_{\lambda} \left( \mathbb{Q} X \right),
\ee
where $\mathbb{S}_{\lambda}$ stands for the Schur functor corresponding to $\lambda$ and $\mathbb{Q}X$ stands for the free vector space on the finite set $X$.
\end{prop}
\noindent
Proposition \ref{schurfree} explains how $\Lambda^k$ and $S^k$---themselves isotype projectives---relate to the classical exterior and symmetric power functors.  Propositions \ref{univiso} and \ref{schurfree} can be deduced by taking $\Bij_k$-isotypic components of earlier observations. \\ \\
\noindent
At the moment, Proposition \ref{univiso} seems a bit disreputable because the isomorphism cannot possibly be canonical: which $\mathbb{P}_{\lambda}$ do we mean?\footnote{Although the definition of $\mathbb{P}_{\lambda}$ depends on a choice of minimal idempotent, the definition of $\Sp_{\lambda}$ depends on the same data, and the problem is resolved.  }  The following definition lends some clarity.
\begin{defn} \label{SchurRS}
The natural inclusion of the symmetric groups $\Bij$ into the category $\Fin$ gives rise to a restriction functor
\be \nn
\mathbb{R} : \Vect^{\Fin} \longrightarrow \Vect^{\Bij}
\ee
with a left adjoint denoted $\mathbb{S}$ satisfying
\be \nn
\Vect^{\Fin}(\mathbb{S}W , \, V)  \; \simeq \;  \Vect^{\Bij}(W ,\, \mathbb{R} V )
\ee
for any representations $W \in \Vect^{\Bij}$ and $V \in \Vect^{\Fin}$.
\end{defn}
\noindent
Considering the Specht module $Sp_{\lambda}$ to be a representation of the category $\Bij$ concentrated on the group $\Bij_k$, we recover a more precise version of Proposition \ref{univiso}.
\begin{defn} \label{schurdef}
A \textbf{Schur projective of degree $k$} is any functor of the form
\be \nn
\mathbb{P} = \mathbb{S}W
\ee
where $W \in \Vect^{\Bij}$ is supported on finite sets of cardinality at most $k$.  If $W$ vanishes on finite sets with cardinality less than $k$, then we say $\mathbb{P}$ is a \textbf{Schur projective of pure degree $k$}.
\end{defn}
\noindent
By writing $W$ as a direct sum of Specht modules, any Schur projective $\mathbb{P}$ is seen to be isomorphic to a direct sum of isotype projectives $\mathbb{P}_{\lambda}$ for various partitions $\lambda$.  The representations $\Theta^k$, $\tensorpower^k$, $\tensorpower^{\leq k}$, and $\tensorpower^{< k}$ are Schur projectives.
\subsection{Finding a good category of representations} \label{goodcategory}
So far, the functor category $\Vect^{\Fin}$ seems perfectly adequate.  However, there is a subtlety we have yet to handle: $\Vect^{\Fin}$ is enriched in $\mathbb{Q}$-vector spaces, but not in \textit{finite dimensional} $\mathbb{Q}$-vector spaces.  Define the functor
\be \nn
\Lambda^{< \infty} = \Lambda^0 \oplus \Lambda^1 \oplus \Lambda^2 \oplus \cdots
\ee
This is a finite-dimensional representation; indeed, $\dim \Lambda^{< \infty}[k] = 2^k$.  However, we will see that its endomorphisms form an infinite dimensional space.  Observing that $\Lambda^{< \infty}$ can be written $\mathbb{S} E$ where $E$ is the direct sum of one copy of every sign representation, we compute
\be \nn
\Vect^{\Fin}(\Lambda^{< \infty} \, , \, \Lambda^{< \infty}) \; \simeq \; \Vect^{\Bij}(E\, , \, \mathbb{R} \Lambda^{< \infty}),
\ee
which is certainly infinite dimensional in light of the sign representation sitting in each $\Lambda^k[k]$.\\ \\
\noindent
For finitely generated $V$, it is immediate that $\Vect^{\Fin}(V, V')$ is finite dimensional, so the enrichment problem is solved by switching to the category of finitely generated representations.
\begin{obs}
The category $\fg$ of finitely generated $\Fin$-modules is enriched in $\Vect$.
\end{obs}
\noindent
The danger, of course, is that the resulting category may lose important homological properties.  For example, we have no guarantee that $\fg$ is closed under taking kernels!  Luckily, we have the following result, which is a consequence of Corollary \ref{polydimgivesfg}.
\begin{thm} \label{abcat}
The category $\fg$ of finitely generated $\Fin$-modules is abelian.  
\end{thm}
\noindent
Finitely generated representations have a notion of degree extending Definition \ref{schurdef}.
\begin{defn}
A finitely generated $\Fin$-representation $V \in \fg$, ``\textbf{has degree} $k$'' if it can be written as a quotient of a Schur projective of degree $k$.  Equivalently, $V$ has degree $k$ if it has a generating set sitting in the vector spaces $V[0], V[1], \ldots, V[k]$.
\end{defn}
\noindent
In order to relate the big representation category $\Vect^{\Fin}$ with the more technically appealing $\fg$, we will prove the following elementary theorem.
\begin{thm} \label{skcolim}
Any representation $V \in \Vect^{\Fin}$ is canonically a colimit of finitely generated subrepresentations
\be \nn
V_0 \subseteq V_1 \subseteq V_2 \subseteq \cdots \subseteq V_k \subseteq \cdots \subseteq V
\ee
where each $V_k$ has degree $k$. See (\ref{skelproof}).
\end{thm}
\noindent
Although this theorem does grant access to the representation theory of $\Fin$ at large, one finds that practical applications lie firmly in $\fg$.
\subsection{Non-semisimplicity of $\Vect^{\Fin}$ and non-polynomiality of dimension}
It is time to expose the reader to some important non-projective representations.  An extremely familiar example turns out to produce a non-split short exact sequence in $\fg$.
\begin{ex}  The free vector space functor $\tensorpower^1$ is a non-trivial extension
\be \nn
0 \longrightarrow D_2 \longrightarrow \tensorpower^1 \longrightarrow D_1 \longrightarrow 0
\ee
Where $D_2$ sends a finite set $X$ to the vectors of $\tensorpower^1X$ with coefficient-sum zero, and $D_1$ is one-dimensional everywhere except for $[0]$, where it vanishes.
\end{ex}
\noindent
\begin{rmk}
Considering finite sets to be discrete topological spaces, the injection $D_2 \hookrightarrow \tensorpower^1$ coincides with the natural inclusion $\tilde H_0^{sing} \longrightarrow H_0^{sing}$ of reduced singular homology with coefficients in $\mathbb{Q}$ into non-reduced singular homology.  
\end{rmk}
\noindent
It turns out that the representations $D_1$ and $D_2$ are part of an infinite family of irreducible representations, which we now construct.
\begin{defn}\label{defpartial}
The \textbf{$k^{th}$ partial} $\partial_k \in \QFin([k],[k+1])$ is defined
\be \nn
\partial_k = \varepsilon_{k} \iota_{k} \varepsilon_{k+1},
\ee
where $\iota_k$ denotes the natural inclusion of $[k]$ into $[k+1]$.
\end{defn}
\noindent
For example, $\partial_2 = \frac{1}{6}(\fn{23} - \fn{32} - \fn{13}+\fn{31} + \fn{12}-\fn{21})$.
\begin{defn} \label{Koszul}
The \textbf{Koszul complex} $\Lambda^{\bullet}$ is the chain complex
\be \nn
\cdots \overset{d_2}{\longrightarrow} \Lambda^2 \overset{d_1}{\longrightarrow} \Lambda^1 \overset{d_0}{\longrightarrow} \Lambda^0 \longrightarrow 0,
\ee
where each map $d_k$ is induced by the corresponding $\partial_k$.
\end{defn}
\begin{obs} Restricting a representation $V \in \Vect^{\Fin}$ to the category of cosimplicial vector spaces $\Vect^{\Delta}$, the chain complex computing reduced cosimplicial homology coincides with $\Vect^{\Fin}(\Lambda^{\bullet}, V)^*$.  In particular, the Koszul complex is exact except at the last term, because evaluating at $[n]$ gives a chain complex computing the reduced homology of an $(n-1)$-simplex.  
\end{obs}
\begin{defn} \label{definedk}
Define
\be \nn
D_k = \coker \left( \Lambda^{k+1} \overset{d_k}{\longrightarrow} \Lambda^k \right),
\ee
the successive cokernels of the maps in the Koszul complex.  Note that $D_0$ is the vector space $\mathbb{Q}$ concentrated on the empty set, and the other $D_k$ are imreps $D_{k} \simeq \langle \partial_{k-1} \rangle$.
\end{defn}
\noindent
Optimistic readers may be hoping that the dimension of a finitely generated representation $V \in \fg$ is given by a polynomial.  Unfortunately, $\dim D_1[n]$ for $n=0,1,2, \ldots$ gives the sequence
\be \nn
0, 0, 1, 2, 3, 4, \ldots
\ee
which is not polynomial.  This example raises the specter of ``eventual polynomiality'' and ``stable ranges'' omnipresent in the representation theory of $\mathcal{FI}$, for example.  However, these fears prove almost entirely unrealized.  In fact, the following computation exemplifies the (extremely limited) extent to which non-polynomiality is available to finitely generated representations; see Corollary \ref{polydim}.
\begin{obs}
An appropriate truncation of the Koszul complex gives a (rightward) resolution of $D_k$, yielding the dimension formula
\be \nn
\dim D_k[n] = \binom{n}{k-1} - \binom{n}{k-2} + \binom{n}{k-3} - \cdots \pm \binom{n}{1} \mp \binom{n}{0} \pm 0^n
\ee
where we take $0^0=1$.  In particular, the dimension of $D_k$ is polynomial away from $[0]$.\end{obs}
\subsection{The representations $C_{\lambda}$}
Sitting inside the representation $\tensorpower^k$ is a subrepresentation called $\sk_{<k} \tensorpower^k$ consisting of the span of all non-injective functions.  Define $C^k$ so as to get a short exact sequence
\be \nn
0 \longrightarrow \sk_{<k} \tensorpower^k \longrightarrow \tensorpower^k \longrightarrow C^k \longrightarrow 0.
\ee
The functor $\tensorpower^n$ has an action of the symmetric group $\Bij_n$ given by permutation of tensors (\`a la Schur-Weyl duality), and $\sk_{<k} \tensorpower^k$ is evidently invariant under this action.  It follows that we get an action of $\Bij_k$ on $C^k$ as well.
\begin{defn} \label{defineclambda}
Given $\lambda$ a partition of $k$, the representation $C_{\lambda}$ is defined
\be \nn
C_{\lambda} = \im\left(C^k \overset{c_{\lambda}}{\longrightarrow} C^k\right)
\ee
where $c_{\lambda}$ is a minimal idempotent of $\mathbb{Q} \Bij_k$ corresponding to the partition $\lambda$, and the action of $\Bij_k$ on $C^k$ is given by permutation of tensors (i.e. precomposition).
\end{defn}
\noindent
Concretely, the representation $C^k$ has a description as a uniformly presented vector space whereby we place generators for all non-injections in the denominator.  For example,
\be \nn
C^4 = \langle \; \fn{1234} \; \rangle \; \big / \; \langle \;  \fn{1123} \; \;\; \; \fn{1213} \; \; \; \;\fn{1231} \; \;\; \; \fn{1223} \;  \; \; \;\fn{1232} \; \;\; \; \fn{1233} \; \rangle.
\ee
A similar trick gives a uniform presentation for any $C_{\lambda}$.  For example, using the Young symmetrizer 
\be \nn
c_{\ydiagram{2,2}} = \frac{1}{12}\left(\scalebox{.7}{$\fn{1234} + \fn{1243}- \fn{1423} - \fn{1432} + \fn{2134} + \fn{2143} - \fn{2314} - \fn{2341} - \fn{3214} - \fn{3241} + \fn{3412} + \fn{3421} - \fn{4123} - \fn{4132} + \fn{4312} + \fn{4321}$} \right),
\ee
we get a corresponding presentation for $C_{\ydiagram{2,2}}$:
\bea \nn
C_{\ydiagram{2,2}} &\simeq & \langle c_{\ydiagram{2,2}} \rangle \; \big / \; \left( \langle c_{\ydiagram{2,2}} \rangle \cap \langle \;  \fn{1123} \; \;\; \; \fn{1213} \; \; \; \;\fn{1231} \; \;\; \; \fn{1223} \;  \; \; \;\fn{1232} \; \;\; \; \fn{1233} \; \rangle \right) \\ \nn
& \simeq & \langle c_{\ydiagram{2,2}} \; \; \; \;  \scalebox{.7}{$\fn{1123} \; \; \; \; \fn{1213} \; \; \; \;\fn{1231} \; \; \; \; \fn{1223} \; \; \; \; \fn{1232} \; \; \; \; \fn{1233}$} \;  \rangle \; \big / \; \langle \;   \scalebox{.7}{$\fn{1123} \; \; \; \; \fn{1213} \; \; \; \;\fn{1231} \; \; \; \; \fn{1223} \; \; \; \; \fn{1232} \; \; \; \; \fn{1233}$} \; \rangle,
\eea
by the second isomorphism theorem.
\begin{obs}
A basis for $C^k[n]$ is given by injections $[k] \hookrightarrow [n]$, where the action of $\Bij_k$ is precomposition.  
\end{obs}
\begin{prop} \label{clambdadim}
\be \nn
\dim C_{\lambda} [n] = \left( \dim \Sp_{\lambda} \right) \cdot \binom{n}{k}.
\ee
\end{prop}
\section{Results}
\subsection{Overview}
Given a finitely generated $\Fin$-representation $V$, our fundamental goal is to compute $\dim V[n]$ as a function of $n$, along with its character as an $\Bij_n$-representation.  Homologically, we'd like to compute kernels, images, cokernels, homs, exts, etc. \\ \\
\noindent
At this point, the most optimistic hope would be to compute a finite projective resolution for $V$.  Unfortunately, we will see later that each $D_k$ has infinite projective dimension. \\ \\
\noindent
As a compromise, we compute finite resolutions using the $D_k$ in addition to Schur projectives.\subsection{The theorems and their corollaries}
\noindent
The following theorem is the main result of the paper.
\begin{thm}[Version A]\label{mainA}
Every uniformly presented vector space has a finite, leftward resolution by sums of Schur projectives and various $D_k$.
\end{thm}
\noindent
We give a more precise version of the main result as Theorem \ref{mainB}.  The next theorem provides a concrete description of finitely generated $\Fin$-representations.
\begin{thm}\label{upvequalsfinrep}
Every uniformly presented vector space is a finitely generated $\Fin$-representation, and vice-versa.
\end{thm}
\noindent
Since the two notions coincide, we will generally make use of the (more convenient) language of $\Fin$-representations.
\begin{thm}\label{fgisfl}
An $\Fin$-representation is finitely generated if and only if it has finite length.
\end{thm}
\begin{cor}
The category $\fg$ has the Jordan-H{\"o}lder property.
\end{cor}
\noindent
These results stand in contrast to the corresponding situation for $\mathcal{FI}$-modules, where finitely generated implies the ascending chain condition, but certainly not finite length.
\noindent
\begin{thm}[Classification of simples] \label{simpleclassification}
Every simple object of $\Vect^{\Fin}$ is isomorphic to exactly one of the following representations:
\be \nn
C_{\ydiagram{2}}, C_{\ydiagram{3}}, C_{\ydiagram{2,1}}, C_{\ydiagram{4}}, C_{\ydiagram{3,1}}, C_{\ydiagram{2,2}}, \ldots
\ee
where the subscript ranges over partitions with at least two boxes in the first row; or
\be \nn
D_0, D_1, D_2, \ldots
\ee
the successive cokernels of the Koszul complex.
\end{thm}
\begin{cor} We have two $\mathbb{Z}$-bases for the $K$-theory of $\fg$.  One is given by the isotype projectives $\mathbb{P}_{\lambda}$ along with the (non-projective) functor $D_0$; the other is given by the simple objects listed above.
\end{cor}
\noindent
This corollary raises an urgent question: what is the transition matrix relating these two bases?
\begin{quest}
How does $\mathbb{P}_{\lambda}$ decompose into simple objects?  How is $C_{\lambda}$ resolved by Schur projectives?
\end{quest}
\begin{cor} \label{polydimgivesfg}
Let $V$ be an $\Fin$-representation with polynomially-bounded dimension
\be \nn
\dim V[n] \leq \psi(n),
\ee
where $\psi(x)$ is a polynomial of degree $d$.   Then $V$ is a quotient of a finite direct sum of tensor powers $\tensorpower^{k}$ with $k \leq d+1$.  In fact, the tensor powers with $k \leq d$ suffice, along with the representation $\Lambda^{d+1}$.
\end{cor}
\begin{cor} \label{tracestats}
Associated to any finitely generated $\Fin$-representation $V$ is a universal trace-computing polynomial
\be \nn
\varphi_{V}(x_0, x_1, x_2, \ldots)
\ee
so that for any endomorphism $f$ of a set of size $n > 0$, the trace of the induced map $Vf$ is given by 
\be \nn
 \Tr( Vf) = \varphi_{V}(n, a_1, a_2, \ldots)
\ee
where $a_i$ is the number of fixed points of $f^i$, the $i^{th}$ iterate of $f$.
\end{cor}
\begin{cor} \label{polydim}
Let $V$ be a finitely generated $\Fin$-representation.  The dimension of $V[n]$ coincides with a polynomial for $n \geq 1$.  Indeed,
\be \nn
\dim V[n] = \varphi_{V}(n, n, n, n, \ldots).
\ee
\end{cor}
\begin{cor}
The restriction of an $\Fin$-representation to the category $\mathcal{FI}$ has extremely limited local cohomology in the sense of Sam and Snowden \cite{GLEquivariantInfinite}.  Since Schur projectives are injective as $\mathcal{FI}$-modules (by \cite{GLEquivariantInfinite} Corollary 4.2.5), all we need to check are the $\Fin$-representations $D_k$.  But these are related by a long exact sequence to the torsion module $D_0$.  It follows that an $\Fin$-representation has vanishing higher local cohomology.  Further, any zeroth local cohomology is concentrated in grade zero.
\end{cor}
\section{General theory and proofs}
\noindent
We expose the general theory of $\Fin$-representations.  Proofs of already-mentioned results begin in Section \ref{minorproofs}.
\subsection{Skeleta}
\begin{defn}[Tensor-Hom adjunction]\label{tensorhom}
The \textbf{tensor product} is a bifunctor
\be \nn
- \otimes - : \Vect \times \Vect^{\Fin} \longrightarrow  \Vect^{\Fin}
\ee
satisfying
\be \nn
 \Vect^{\Fin}(Z \otimes V, V') \simeq \Vect(Z, \;  \Vect^{\Fin}(V, V'))
\ee
for any finite dimensional vector space $Z$ and any pair of $\Fin$-representations $V$ and  $V'$ with $V$ finitely generated.  The counit of this adjunction is called \textbf{evaluation}.  The one-dimensional vector space $\mathbb{Q}$ is the (left) identity for the tensor product.  See \cite{CategoricalHomotopyTheory}.
\end{defn}
\begin{defn}[Skeleta] \label{skel}
Given representations $V$ and $P$ with $P$ finitely generated, the \textbf{$P$-skeleton of $V$} is defined
\be \nn
\sk_P V = \im \left(  \Vect^{\Fin}(P, V) \otimes P \overset{ev}{\longrightarrow} V \right)
\ee
the image of the evaluation map.  For convenience, the $\tensorpower^{\leq k}$-skeleton is called the $k$-skeleton hereafter and written
\be \nn
\sk_{\leq k} = \sk_{\scalebox{.7}{$\tensorpower^{\leq k}$}}.
\ee
Similarly, set
\be \nn
\sk_{< k} = \sk_{\scalebox{.7}{$\tensorpower^{< k}$}}.
\ee
Naturality of the counit makes $\sk_P$ into a functor $\sk_P : \UP \longrightarrow \Vect$;
furthermore, we have a canonical natural transformation called the \textbf{skeletal inclusion}
\be \nn
\sk_P V \hookrightarrow V
\ee
which is componentwise injective.
\end{defn}
\noindent
Concretely, the $k$-skeleton of a representation $V$ is the smallest subrepresentation that coincides with $V$ when evaluated on finite sets of size at most $k$.  Equivalently, any vector $v \in \sk_k V[n]$ in the $k$-skeleton of $V$ is in the $\Fin$-span of the vectors living in degree $k$ or below.\\ \\
\noindent
In light of these observations, we pause to offer the proof of Theorem \ref{skcolim}.
\begin{proof} 
Use the skeletal filtration
\be \label{skelproof}
\sk_{\leq 0} V \subseteq \sk_{\leq 1} V \subseteq \cdots \subseteq \sk_{\leq k} V \subseteq \cdots \subseteq V.
\ee
\end{proof}
\begin{obs} \label{calledB}
An injection of representations $U \hookrightarrow V$ induces an injection on skeleta $\sk_P U \hookrightarrow \sk_P V$.
\end{obs}
\begin{obs} \label{calledA}
Let $U$, $V$, and $P$ be representations.  The canonical inclusion
\be \nn
0 \longrightarrow \sk_P V \longrightarrow V
\ee
induces an isomorphism
\be \nn
0 \longrightarrow  \Vect^{\Fin} \left( \sk_P U, \sk_P V \right) \longrightarrow  \Vect^{\Fin} \left( \sk_P U, V \right) \longrightarrow 0.
\ee
In other words, any map $f: \sk_P U \longrightarrow V$ lands inside the $P$-skeleton of $V$.
\end{obs}
\begin{prop} \label{skcounit}
Let $\mathcal{G}$ be a wide subcategory of $\Fin_{\leq k}$.  The $k$-skeleton of an $\Fin$-representation coincides with the image of the counit of the adjunction $\Lan_{\mathcal{G}}^{\Fin} \dashv \Res_{\mathcal{G}}^{\Fin}$.
\end{prop}
\begin{proof}
Let $\mathcal{H}$ denote the minimal wide subcategory of $\Fin_{\leq k}$, that is, the subcategory consisting of the sets $[0], [1], \ldots, [k]$ and only identity arrows.  By transitivity of Kan extensions, the counit $\epsilon_{\mathcal{H}}^{\Fin} : \Lan_{\mathcal{H}}^{\Fin} \Res_{\mathcal{H}}^{\Fin} \longrightarrow 1$ factors through the two intermediate counits
\be \nn
\Lan_{\mathcal{H}}^{\Fin} \Res_{\mathcal{H}}^{\Fin} \overset{\sim}{\longrightarrow} \Lan_{\mathcal{G}}^{\Fin} \Lan_{\mathcal{H}}^{G} \Res_{\mathcal{H}}^{G} \Res_{\mathcal{G}}^{\Fin} \longrightarrow \Lan_{\mathcal{G}}^{\Fin} \Res_{\mathcal{G}}^{\Fin} \longrightarrow 1.
\ee
The functor $\Res^{\mathcal{G}}_{\mathcal{H}}$ is faithful, so the counit $\epsilon^{\mathcal{G}}_{\mathcal{H}}$ is componentwise epic.  It follows from right exactness of $\Lan_{\mathcal{G}}^{\Fin}$ that the middle map is a surjection and $\im(\epsilon_{\mathcal{H}}^{\Fin}) = \im(\epsilon_{\mathcal{G}}^{\Fin})$.  Since $\mathcal{H} = \bigsqcup_{i\leq k} [i]$, and
\be \nn
\Res^{\Fin}_{[i]} \simeq \Vect^{\Fin}(\tensorpower^i, - )
\ee
are isomorphic functors, uniqueness of left adjoints tells us $\im(\epsilon_{\mathcal{H}}^{\Fin}) \simeq \im(ev_{\scalebox{.5}{$\tensorpower^{\leq k}$}}) \simeq \sk_{\leq k} $ as required.
\end{proof}

\subsection{Squishing}
The following crucial definition combined with the Squishing Lemmas \ref{uppersquisher} and \ref{lowersquisher} form the crux of the proof of the main result, Theorem \ref{mainA}.
\begin{defn}
Given representations $V$, $S$, and $P$, we say 
\be \nn
\mbox{\textbf{``\,$V$ squishes $S$ through $P$''}}
\ee
if there exists some finite-dimensional vector space $Z$ and maps
\be \nn
S \longrightarrow Z \otimes P \longrightarrow S
\ee
called \textbf{squishing maps} with the property that the induced composition
\be \nn
 \Vect^{\Fin}(S, V) \longleftarrow  \Vect^{\Fin}(Z \otimes P, V) \longleftarrow  \Vect^{\Fin}(S,V)
\ee
is the identity map.
\end{defn}
\begin{defn}[Triple Hom]\label{triplehom}
Let $\mathcal{C}$ be an additive category.  For objects $X$, $Y$, $Z$ in $\mathcal{C}$, define
\be \nn
\mathcal{C}(X, Y, Z) = \im \left( \mathcal{C}(X, Y) \otimes_{\mathbb{Z}} \mathcal{C}(Y, Z) \longrightarrow \mathcal{C}(X, Z) \right)
\ee
the image of the composition map.  Concretely, $\mathcal{C}(X, Y, Z)$ is the span of all maps from $X$ to $Z$ that factor through a map to $Y$.
\end{defn}
\begin{obs} \label{twosided}
Given objects $X, Y \in \mathcal{C}$,
\be \nn
\mathcal{C}(X,Y,X) \subseteq \mathcal{C}(X, X)
\ee
is the inclusion of a two-sided ideal.
\end{obs}
\begin{defn}
Given representations $S$ and $P$, a \textbf{squisher} for a representation $V$ is an element $\omega \in \UP(S, P, S)$ with the property that $\UP(\omega, V) $ is the identity map on $\UP(S,V)$.  Tautologously, $V$ squishes $S$ through $P$ exactly when $V$ has a squisher.
\end{defn}
\begin{obs}[Transitivity of squishing] \label{transitivesquish}
If $V$ squishes $A$ through $B$ and $V$ also squishes $B$ through $C$, then $V$ squishes $A$ through $C$.
\end{obs}
\begin{obs} \label{sumsquish}
For fixed $V$ and $P$, the set 
\be \nn
\{ S  : \, \mbox{$V$ squishes $S$ through $P$} \}
\ee
 is closed under taking direct sums, taking direct summands, and tensoring with a finite-dimensional vector space.
\end{obs}
\begin{obs} \label{squishinnersum}
For fixed $V$ and $S$, the set 
\be \nn
\{ P  :  \, \mbox{$V$ squishes $S$ through $P$} \}
\ee
 is closed under taking direct sums.
\end{obs}
\begin{lem}[Extensions Squish] \label{squishext}
For fixed $S$ and $P$ with $S$ projective, the set 
\be \nn
\{ V  : \, \mbox{$V$ squishes $S$ through $P$} \}
\ee
 is closed taking subrepresentations, taking quotients, and taking extensions.
\end{lem}
\begin{proof}
Given a short exact sequence of representations
\be \nn
0 \longrightarrow A \longrightarrow B \longrightarrow C \longrightarrow 0
\ee
we must show that $A$ and $C$ both squish $S$ through $P$ if and only if $B$ squishes $S$ through $P$.  Suppose $B$ squishes $S$ through $P$, witnessed by some squisher $\omega \in \UP(S, P, S)$.  Build the diagram
\be \nn
\begin{xy}
(0,20)*+{0}="a"; (30,20)*+{\UP(S, A)}="b"; (60,20)*+{\UP(S, B)}="c"; (90,20)*+{\UP(S, C)}="d"; (120,20)*+{0}="e"; 
(0,0)*+{0}="f"; (30,0)*+{\UP(S, A)}="g"; (60,0)*+{\UP(S, B)}="h"; (90,0)*+{\UP(S, C)}="i"; (120,0)*+{0}="j";
{\ar "a"; "b"};
{\ar "b"; "c"};
{\ar "c"; "d"};
{\ar "d"; "e"};
{\ar "f"; "g"};
{\ar "g"; "h"};
{\ar "h"; "i"};
{\ar "i"; "j"};
{\ar "b"; "g"};
{\ar_1 "c"; "h"};
{\ar "d"; "i"};
\end{xy}
\ee
where the rows are still exact since $\UP(S,-)$ is exact, and the vertical maps are induced by $\omega$.  By construction, the middle map is the identity.  It follows that the outer two vertical maps are also identity maps.  We conclude that $\omega$ is a squisher for both $A$ and $C$ as well. \\
\\
\noindent
In the other direction, let $\omega_A, \omega_C \in \UP(S, P, S)$ be squishers for $A$ and $C$, respectively:
\bea \nn
\UP(\omega_A, A) & = & 1 \\ \nn
\UP(\omega_C, C) & = & 1.
\eea
Define $\nu = \omega_A - \omega_A \omega_C + \omega_C$, noting that $\nu \in \UP(S, P, S)$ by Observation \ref{twosided}.  We know that $\UP(1-\omega_A, A)=0$ and $\UP(1-\omega_C, C)=0$, so the composite $(1-\omega_A) (1-\omega_C) = 1 - \nu$ gives zero in both cases:
\bea \nn
\UP(1-\nu, A) & = & 0 \\ \nn
\UP(1-\nu, C) & = & 0.
\eea
It follows that $\nu$ is a squisher for $A$ and $C$ simultaneously.  As before, build the diagram of induced endomorphisms
\be \nn
\begin{xy}
(0,20)*+{0}="a"; (30,20)*+{\UP(S, A)}="b"; (60,20)*+{\UP(S, B)}="c"; (90,20)*+{\UP(S, C)}="d"; (120,20)*+{0}="e"; 
(0,0)*+{0}="f"; (30,0)*+{\UP(S, A)}="g"; (60,0)*+{\UP(S, B)}="h"; (90,0)*+{\UP(S, C)}="i"; (120,0)*+{0}="j";
{\ar "a"; "b"};
{\ar "b"; "c"};
{\ar "c"; "d"};
{\ar "d"; "e"};
{\ar "f"; "g"};
{\ar "g"; "h"};
{\ar "h"; "i"};
{\ar "i"; "j"};
{\ar_1 "b"; "g"};
{\ar "c"; "h"};
{\ar^1 "d"; "i"};
\end{xy}
\ee
where the vertical maps are induced by $\nu$.  This time the outer two vertical maps are identity maps.  The 5-lemma forces the middle map $e=\UP(\nu, B)$ to be an isomorphism.  Let $\varphi_e(x)$ be the characteristic polynomial of $e$, and define
\be \nn
\psi(x) = 1- \frac{\varphi_e(x)}{\det(e)}.
\ee
Note that $\det(e) \neq 0$ since $e$ is invertible and that $\psi(x)$ has no constant term.  Let
$\omega_B = \psi(\nu)$, observing that $\omega_B \in \UP(S,P,S)$, once again by Observation \ref{twosided}.  Compute
\bea \nn
\UP(\omega_B, B) & = & \UP(\psi(\nu), B) \\ \nn
& = & \psi( \UP(\nu, B)) \\ \nn
& = & \psi(e) \\ \nn
& = & 1 - \frac{\varphi_e(e)}{\det(e)} \\
& = & 1 - \frac{0}{\det(e)}   \label{cayham} \\ \nn
& = & 1,
\eea
where we used the Cayley-Hamilton theorem in (\ref{cayham}).  It follows that $\omega_B$ is a squisher for $B$ and we are done.
\end{proof}
\begin{obs} \label{skelsquish}
If $P$ squishes $S$ through $P$, then $\sk_P V$ squishes $S$ through $P$ as well for any representation $V$.  Indeed, $\sk_P V$ is defined as a quotient of an extension (actually, a direct sum) of copies of $P$.
\end{obs}
\subsection{Squishers for $\Fin$}
\noindent
The following proposition prepares us to find squishers for $\Fin$.
\begin{prop} \label{missedfactor}
Let $f: [p] \longrightarrow [q]$ be a function, and let $v_1, v_2, \ldots, v_q \in \mathbb{Q}[q]$ be vectors.  Define
\be \nn
v = v_1 \otimes v_2 \otimes \cdots \otimes v_q \in \otimes^q \left( \mathbb{Q}[q] \right) \simeq \QFin([q],[q])
\ee
where the isomorphism of vector spaces is spelled out in Example \ref{oneline}.  If the function $f$ does not hit $k \in [q]$ in its image, and if $v_k$ has coefficient-sum zero, then $fv=0$. 
\end{prop}
\begin{proof}
We begin with the case where each $v_i$ is a basis vector for $i \neq k$.  Suppose $v_k = \sum \alpha_j e_j$ a linear combination of basis vectors with $\sum \alpha_j = 0$.  Expanding,
\bea \nn
v & = & v_1 \otimes v_2 \otimes \cdots \otimes v_k \otimes \cdots \otimes v_q \\ \nn
  & = & v_1 \otimes v_2 \otimes \cdots \otimes  \left(\sum \alpha_j e_j \right) \otimes \cdots \otimes v_q \\ \nn
  & = & \sum \alpha_j \left( v_1 \otimes v_2 \otimes \cdots \otimes  e_j \otimes \cdots \otimes v_q\right)
\eea
Since the image of $f$ omits $k$, each of these tensors---when interpreted as a function---maps to the same function under precomposition with $f$.  We know that $\sum \alpha_j = 0$, so $fv=0$ as required.  The general case follows by linearity.
\end{proof}
\begin{lem}[Upper Squishing Lemma for $\Fin$]\label{uppersquisher}
The projective $\tensorpower^k$ squishes $\tensorpower^n$ through $\tensorpower^{\leq k+1}$ for all $n$.
\end{lem}
\begin{proof}
 By transitivity of squishing Observation \ref{transitivesquish} it suffices to show that $\tensorpower^k$ squishes $\tensorpower^n$ through $\tensorpower^{<n}$.  There is nothing to prove unless $n > k+1$.   In this case, form the vector $\nu =$
\be \nn
\underbrace{e_1 \hspace{-3pt} \otimes \hspace{-2pt} e_2 \hspace{-3pt} \otimes \hspace{-2pt} \cdots \hspace{-3pt} \otimes \hspace{-2pt} e_{n-k-1}}_{\mbox{\scalebox{.75}{$n-(k+1)$ factors}}} \otimes \underbrace{(e_{n-k} - e_{n-k-1}) \hspace{-3pt} \otimes \hspace{-3pt} (e_{n-k+1} - e_{n-k}) \hspace{-3pt} \otimes \hspace{-3pt} (e_{n-k+2} - e_{n-k+1}) \hspace{-3pt} \otimes \hspace{-3pt} \cdots \hspace{-3pt} \otimes \hspace{-3pt} (e_{n} - e_{n-1})}_{\mbox{\scalebox{.75}{$(k+1)$ factors}}}
\ee
which is an element of $\otimes^{n} \left( \mathbb{Q}[n] \right) \simeq \QFin([n],[n])$.  By inspection, this vector is congruent to the identity modulo the two-sided ideal generated by the non-bijective functions:
\bea \nn
\QFin([n], [n]) &\twoheadrightarrow& \mathbb{Q} \Bij_n \\ \nn
\nu &\mapsto& 1.
\eea
However, $\nu$ annihilates every function in $\Fin([k],[n])$ by Proposition \ref{missedfactor} since every such function misses at least one of the last $(k+1)$ tensor factors, which have zero coefficient-sum.  It follows that $1-\nu$ is the required squisher.
\end{proof}
\begin{lem}[Lower Squishing Lemma for $\Fin$]\label{lowersquisher}
The projective $\tensorpower^k$ squishes $\Theta^{k+1}$ through $\tensorpower^{\leq k}$.
\end{lem}
\begin{proof}
When $k=0$, $\Theta^{k+1} = 0$ and there is nothing to prove.  Assuming $k \geq 1$, form the vector
\be \nn
\mu = (e_1 - e_2) \otimes (e_2 - e_1) \otimes (e_3 - e_2) \otimes (e_4 - e_3) \otimes (e_5 - e_4) \otimes \cdots \otimes (e_{k+1} - e_k),
\ee
where the pattern begins in earnest after the first tensor factor.  This vector is congruent to $1 + (1\;2) \in \mathbb{Q} \Bij_{k+1}$ modulo non-bijections.  The only representations of $\Bij_{k+1}$ sending $1 + (1\;2)$ to zero are multiples of the sign representation---indeed, any representation sending $(1\;2)$ to $-1$ must also send every other transposition to $-1$.  It follows by the Artin-Wedderburn theorem that $1 - \varepsilon \in \mathbb{Q} \Bij_{k+1}$ is in the two-sided ideal generated by $1 + (1\;2)$.  On the other hand, it is clear by Proposition \ref{missedfactor} that $\mu$ annihilates the entire vector space $\QFin([k],[k+1])$.  Any other vector in the two-sided ideal generated by $\mu$ will also have this property, and so we may find some such vector $\mu'$ projecting down to $1-\varepsilon$.  This time, $1-\varepsilon - \mu'$ is the required squisher.
\end{proof}
\subsubsection{Example squishers for $\Fin$}
Due to the importance of the squishing lemmas, we give the reader an idea of the squishers they produce.  When $k=2$ and $n=4$, the Upper Squishing Lemma produces the element
\be \nn
\fn{1123} - \fn{1124} - \fn{1133} + \fn{1134} - 
 \fn{1223} + \fn{1224} + \fn{1233}
\ee
which induces an identity map after application of the functor $\tensorpower^2$.  Concretely, any function from $[2]$ to $[4]$ is fixed under right multiplication by this element.  For example, 
\bea \nn
\left( \fn{14} \right) \left(\fn{1123} - \fn{1124} - \fn{1133} + \fn{1134} - 
 \fn{1223} + \fn{1224} + \fn{1233} \right) & = & \fn{13} - \fn{14} - \fn{13} + \fn{14} - 
 \fn{13} + \fn{14} + \fn{13} \\ \nn
 &=& \fn{14}.
\eea
The Lower Squishing Lemma produces messier vectors.  In the case $k=2$,
\bea \nn
&& -\frac{1}{6} \fn{111}-\frac{1}{4} \fn{112}+\frac{5}{12} \fn{113}+\frac{5}{12} \fn{121}+\frac{5}{12} \fn{122}-\frac{1}{4} \fn{131}  +\frac{5}{12} \fn{133}+\frac{1}{12} \fn{211} \\ \nn
&&\hspace{1in} +\frac{1}{12} \fn{212}-\frac{1}{4} \fn{221}-\frac{1}{6} \fn{222}+\frac{5}{12} \fn{223}+\frac{1}{12} \fn{232}-\frac{1}{4} \fn{233} +\frac{1}{12} \fn{311} \\ \nn
&&\hspace{1in}-\frac{1}{4} \fn{313}-\frac{1}{4} \fn{322}+\frac{5}{12} \fn{323}+\frac{1}{12} \fn{331}+\frac{1}{12} \fn{332}-\frac{1}{6} \fn{333}
\eea
acts by right multiplication on functions from $[2]$ to $[3]$ in the same way that $1-\varepsilon$ acts on functions from $[2]$ to $[3]$.  For example, $\fn{13}+\fn{31}$ is killed by $\varepsilon$ and so is fixed by this element.
\subsection{Squishing constrains skeleta}

\begin{prop}\label{doubleskconclusion}
Suppose $U$ squishes $S$ through $P$, with $S \in \fg$.  The skeletal injection 
\be \nn
\sk_P \sk_S U \hookrightarrow \sk_S U
\ee
 is actually an isomorphism.
\end{prop}
\begin{proof}
By Lemma \ref{squishext}, $\sk_S U$ squishes $S$ through $P$ since it is subobject of $U$.  It follows that there exists a finite-dimensional vector space $Z$ and maps
\be \nn
S \overset{\alpha}{\longrightarrow} Z \otimes P \overset{\beta}{\longrightarrow} S
\ee
with $\UP(\alpha \beta, \sk_S U) = 1$.  For any map $f: S \longrightarrow \sk_S U$, build the diagram
\be \nn
\begin{xy}
(20,30)*+{S}="b"; (60,30)*+{Z \otimes P}="c"; (100,30)*+{\sk_P \sk_S U}="d";
(60,15)*+{S}="cc";
(20,0)*+{\sk_S U}="g"; (60,0)*+{\sk_S U}="h"; (100,0)*+{\sk_S U}="i";
{\ar^{\alpha} "b"; "c"};
{\ar "c"; "d"};
{\ar_1 "g"; "h"};
{\ar_1 "h"; "i"};
{\ar_f "b"; "g"};
{\ar_{\beta} "c"; "cc"};
{\ar_f "cc"; "h"};
{\ar "d"; "i"};
\end{xy}
\ee
where the first square commutes since $\alpha \beta$ is a squisher and the second square is a consequence of Observation \ref{calledA}.  We see that any $f$ factors through the skeletal inclusion $\sk_P \sk_S U \hookrightarrow \sk_S  U $, so the induced map
\be \nn
\UP(S, \sk_P \sk_S  U ) \overset{\sim}{\longrightarrow} \UP(S, \sk_S  U )
\ee
is an isomorphism.  Applying the functor $\Vect\left(\UP(S,  U ), -\right)$, we obtain
\be \nn
\Vect\left(\UP(S,  U ), \UP(S, \sk_P \sk_S  U )\right) \overset{\sim}{\longrightarrow} \Vect\left(\UP(S,  U ),\UP(S, \sk_S  U )\right).
\ee
Naturality of the tensor-hom adjunction from Definition \ref{tensorhom} gives a corresponding isomorphism
\be \nn
\UP\left(\UP(S,  U ) \otimes S, \sk_P \sk_S  U \right) \overset{\sim}{\longrightarrow} \UP\left(\UP(S,  U ) \otimes S, \sk_S  U \right).
\ee
The vector space on the right contains the surjection $\UP(S,  U ) \otimes S \longrightarrow \sk_S  U $, and so there exists a corresponding vector on the right
\be \nn
\UP(S,  U ) \otimes S \longrightarrow \sk_P \sk_S  U 
\ee
with the property that post-composing with the skeletal inclusion gives the canonical surjection onto the $S$-skeleton of $ U $:
\be \nn
\begin{xy}
(0,10)*+{\UP(S,  U ) \otimes S}="a"; (25,0)*+{\sk_P \sk_S  U }="b"; (50,10)*+{\sk_S  U }="c"; 
{\ar "a"; "b"};
{\ar "b"; "c"};
{\ar "a"; "c"};
\end{xy}
\ee
The composite map is surjective, so skeletal inclusion is also surjective, and hence an isomorphism.
\end{proof}

\begin{prop}\label{doubleskhypothesis}
Suppose the natural injection $\sk_P \sk_S U \hookrightarrow \sk_S U$ is an isomorphism.  Any composite $S \longrightarrow U \longrightarrow V$ factors through the natural injection $\sk_P V \hookrightarrow V$:
\be \nn
\begin{xy}
(20,20)*+{S}="b"; (50,20)*+{\sk_P V}="c";
(20,0)*+{U}="g"; (50,0)*+{V.}="h";
{\ar "b"; "c"};
{\ar "g"; "h"};
{\ar "b"; "g"};
{\ar "c"; "h"};
\end{xy}
\ee
\end{prop}
\begin{proof}
Build the diagram
\be \nn
\begin{xy}
(5,20)*+{S}="a"; (30,20)*+{\sk_S S}="b"; (60,20)*+{\sk_S U}="c"; (90,20)*+{\sk_P \sk_S U}="e"; (120,20)*+{\sk_P V}="ee"; 
(5,0)*+{U}="f"; (30,0)*+{U}="g"; (60,0)*+{U}="h"; (90,0)*+{V}="j"; (120,0)*+{V,}="jj";
{\ar^{\sim} "a"; "b"};
{\ar "b"; "c"};
{\ar_{\sim} "e"; "c"};
{\ar "e"; "ee"};
{\ar_1 "f"; "g"};
{\ar_1 "g"; "h"};
{\ar_1 "j"; "jj"};
{\ar "a"; "f"};
{\ar "b"; "g"};
{\ar "c"; "h"};
{\ar "h"; "j"};
{\ar "e"; "j"};
{\ar "ee"; "jj"};
\end{xy}
\ee
where the second and last commuting squares use Observation \ref{calledA}.
\end{proof}

\begin{lem}\label{calledCombo}
If $U$ squishes $S$ through $P$ then any composite $S \longrightarrow U \longrightarrow V$ factors
\be \nn
\begin{xy}
(20,20)*+{S}="b"; (50,20)*+{\sk_P V}="c";
(20,0)*+{U}="g"; (50,0)*+{V,}="h";
{\ar "b"; "c"};
{\ar "g"; "h"};
{\ar "b"; "g"};
{\ar "c"; "h"};
\end{xy}
\ee
where the map on the right is the skeletal inclusion.
\end{lem}
\begin{proof}
Combine Propositions \ref{doubleskconclusion} and \ref{doubleskhypothesis}.
\end{proof}
\begin{prop} \label{intersectlem}
Let $U \hookrightarrow V$ be an injection, and suppose $P$ squishes $S$ through $P$, with $S$ finitely generated.  We have an equality
\be \nn
\UP\left(S, (\sk_P V) \cap U\right) = \UP(S, \sk_P U).
\ee
\end{prop}
\begin{proof}
The $L$-shaped diagram
\be \nn
\begin{xy}
(20,20)*+{\UP(S, (\sk_P V) \cap U) \otimes S}="b"; (60,20)*+{\phantom{\sk_P U}}="c";
(20,0)*+{(\sk_P V) \cap U}="g"; (65,0)*+{U}="h";
{\ar "g"; "h"};
{\ar_{ev} "b"; "g"};
\end{xy}
\ee
extends to a commuting square by Lemma \ref{squishext}, Observation \ref{sumsquish}, and Lemma \ref{calledCombo}:
\be \nn
\begin{xy}
(20,20)*+{\UP(S, (\sk_P V) \cap U) \otimes S}="b"; (65,20)*+{\sk_P U}="c";
(20,0)*+{(\sk_P V) \cap U}="g"; (65,0)*+{U.}="h";
{\ar "b"; "c"};
{\ar "g"; "h"};
{\ar_{ev} "b"; "g"};
{\ar "c"; "h"};
\end{xy}
\ee
Applying the functor $\UP(S,-)$,
\be \nn
\begin{xy}
(20,20)*+{\UP\left(S, \UP(S, (\sk_P V) \cap U) \otimes S\right)}="b"; (65,20)*+{\UP(S, \sk_P U)}="c";
(20,0)*+{\UP(S, (\sk_P V) \cap U)}="g"; (65,0)*+{\UP(S, U).}="h";
{\ar "b"; "c"};
{\ar "g"; "h"};
{\ar_{\UP(S,ev)} "b"; "g"};
{\ar "c"; "h"};
\end{xy}
\ee
The left map is (split) epic by the counit-unit equations; the right and bottom maps remain injective by left-exactness of $\UP(S,-)$.  Epi-mono factorization gives a lift
\be \nn
\begin{xy}
(20,20)*+{\UP\left(S, \UP(S, (\sk_P V) \cap U) \otimes S\right)}="b"; (65,20)*+{\UP(S, \sk_P U)}="c";
(20,0)*+{\UP(S, (\sk_P V) \cap U)}="g"; (65,0)*+{\UP(S, U).}="h";
{\ar "b"; "c"};
{\ar "g"; "h"};
{\ar_{\UP(S,ev)} "b"; "g"};
{\ar "c"; "h"};
{\ar "g"; "c"};
\end{xy}
\ee
which is forced to be an inclusion since the bottom map is.  We have established the forward inclusion $\UP(S, (\sk_P V) \cap U) \subseteq \UP(S, \sk_P U)$.  In the other direction,
\be \nn
\UP(S, (\sk_P V) \cap U) = \UP(S, \sk_P V) \cap \UP(S, U)
\ee
by left exactness of $\UP(S, -)$, so it suffices to prove the inclusions
\be \nn
\UP(S, \sk_P U) \subseteq \UP(S, \sk_P V) \hspace{.5in} \mbox{and} \hspace{.5in} \UP(S, \sk_P U) \subseteq \UP(S, U).
\ee
Observation \ref{calledB} gives that the induced map $\sk_P U \hookrightarrow \sk_P V$ is an injection, after which both statements follow from the left exactness of $\UP(S, -)$.
\end{proof}
\subsection{Homology}
Following \cite{FIMod}, we define a notion of zeroth homology.
\begin{defn} \label{zerothhomology}
Given a representation $V : \Fin \longrightarrow \Vect$, its \textbf{zeroth homology in degree k} is defined
\bea \nn
H_0V[k] &=&  \left(V \Big / \sk_{< k} V  \right) [k]  \\
& = & \coker \left(\Res_{\Bij_k}^{\Fin} \Lan_{\Bij_{< k}}^{\Fin} \Res_{\Bij_{< k}}^{\Fin} V\overset{\scalebox{.5}{\hspace{-15pt}$\Res \epsilon$}}{\longrightarrow} \Res_{\Bij_k}^{\Fin} V\right), \label{degkhom}
\eea
where the map is the restriction of the counit
\be \nn
\Res_{\Bij_k}^{\Fin} \epsilon_{\Bij_{<k}}^{\Fin}.
\ee
The two definitions are equivalent by Proposition \ref{skcounit}.  We define the \textbf{zeroth homology functor}
\bea \nn
H_0 : \Bij & \longrightarrow & \Vect \\ \nn
\left[k\right] & \longmapsto & H_0V[k]
\eea
using the universal property of the coproduct $\sqcup \Bij_k \simeq \Bij$.  The \textbf{homology map h} is defined to be the natural surjection
\be \nn
\mathbb{R} V \overset{h}{\longrightarrow} H_0V \longrightarrow 0
\ee
attached to the various cokernels from (\ref{degkhom}).
\end{defn}
\begin{obs}\label{homobs}
The homology map $h$ is natural in $V$.  The functor $H_0$ is right exact and satisfies
\be \nn
\Res^{\Bij}_{\Bij_0} H_0 = \Res^{\Bij}_{\Bij_0}.
\ee
\end{obs}
\begin{rmk}[Higher homology]
Since $H_0$ is right exact, is has left derived ``higher homology'' functors.  Although higher homology plays no role in the proofs of this paper, it may nevertheless be valuable to think of $H_0 V$ as ``generators of $V$'' and higher homology as ``syzygies of $V$.''
\end{rmk}
\begin{prop}[Zeroth homology and squishing] \label{killhomology}
Let $\mathbb{P} = \Lan_{\Bij_{k}}^{\Fin} W$ be a Schur projective of (pure) degree $k$, and suppose some representation $V$ squishes $\mathbb{P}$ through $\tensorpower^{<k}$.  Then 
\be \nn
\Vect^{\Bij_k}( W , \;  H_0 V[k]) = 0.
\ee
\end{prop}
\begin{proof}
We have a short exact sequence
\be \nn
0 \longrightarrow \sk_{< k} V[k] \longrightarrow V[k] \longrightarrow H_0 V[k] \longrightarrow 0.
\ee
Applying the exact functor $\Vect^{\Bij_k}(W, -)$:
\be \nn
0 \longrightarrow \Vect^{\Bij_k}\left(W,\; \sk_{< k} V[k] \right) \longrightarrow \Vect^{\Bij_k}\left(W,\; V[k]\right) \longrightarrow \Vect^{\Bij_k}\left(W, \; H_0 V[k] \right) \longrightarrow 0.
\ee
In what follows, let $M = \Vect^{\Bij_k}\left(W, \; H_0 V[k] \right)$; we wish to show $M$ is zero.  Let $W' \in \Vect^{\Bij}$ be the functor sending $[k]$ to $W$ but other sets to zero.  We have
\be \nn
0 \longrightarrow \Vect^{\Bij}\left(W',\; \mathbb{R} \, \sk_{< k} V \right) \longrightarrow \Vect^{\Bij}\left(W',\; \mathbb{R}V\right) \longrightarrow M \longrightarrow 0.
\ee
The adjunction $\mathbb{S} \dashv \mathbb{R}$ gives
\be \nn
0 \longrightarrow \UP(\mathbb{S} W', \sk_k V) \longrightarrow\UP(\mathbb{S} W', V) \longrightarrow M \longrightarrow 0.
\ee
By Observation \ref{sumsquish}, $V$ squishes $\mathbb{S} W'$ through $\tensorpower^{<k}$ since $\mathbb{S} W'$ is a direct summand of $\mathbb{S} W$.  By Lemma \ref{calledCombo}, any map from $\mathbb{S} W'$ to $V$ lands inside the $k$-skeleton.  It follows that $M$ vanishes and the claim is proved.
\end{proof}
\begin{rmk}
The next few intermediate results Proposition \ref{homreflectssurjections}, Proposition \ref{schurhomology}, Proposition \ref{schurhomologyindegree}, and Lemma \ref{goodcover} are analogous to ones found in \cite{FIMod}, where Church, Ellenberg, and Farb give essentially one-line proofs.  The results are no harder for $\Fin$-representations, but our formalism causes the proofs to swell a bit.  Nevertheless, we maintain our more categorical approach in these proofs, both for self-consistency, and with the understanding that a more formal style sometimes pays dividends later on.
\end{rmk}
\begin{prop}\label{homreflectssurjections}
Any map $V \longrightarrow W$ inducing a surjection
\be \nn
H_0 V \longrightarrow H_0 W \longrightarrow 0
\ee
is itself a surjection.
\end{prop}
\begin{proof}
It suffices to prove a surjection
\be \nn
\Res_{\Bij_k}^{\Fin} V \longrightarrow \Res_{\Bij_k}^{\Fin} W \longrightarrow 0
\ee
for each $k$.  When $k=0$, Observation \ref{homobs} gives us the surjection.  Proceeding by induction, we may assume a surjection
\be \nn
\Res_{\Bij_{< k}}^{\Fin} V \longrightarrow \Res_{\Bij_{< k}}^{\Fin} W \longrightarrow 0
\ee
in order to prove the surjection for $k$.  Build the diagram defining the induced map on zeroth homology in degree $k$:
\be \nn
\begin{xy}
(20,20)*+{\Res_{\Bij_k}^{\Fin} \Lan_{\Bij_{<k}}^{\Fin} \Res_{\Bij_{<k}}^{\Fin} V}="a";
(20,0)*+{\Res_{\Bij_k}^{\Fin} \Lan_{\Bij_{<k}}^{\Fin} \Res_{\Bij_{<k}}^{\Fin} W}="b";
(70,20)*+{\Res_{\Bij_k}^{\Fin} V}="c";
(70,0)*+{\Res_{\Bij_k}^{\Fin} W}="d";
(100,20)*+{H_0V[k]}="e";
(100,0)*+{H_0W[k]}="f";
(130,20)*+{0}="g";
(130,0)*+{0}="h";
{\ar "a"; "b"};
{\ar "c"; "d"};
{\ar "e"; "f"};
{\ar "a"; "c"};
{\ar^h "c"; "e"};
{\ar "e"; "g"};
{\ar "b"; "d"};
{\ar_h "d"; "f"};
{\ar "f"; "h"};
\end{xy}
\ee
The first vertical map is surjective since by the inductive hypothesis combined with right exactness of $\Res$ and $\Lan$.  The last vertical map is surjective by assumption.  Replacing the lower left corner with a suitable image, we can arrange for the second row to be a short exact sequence:
\be \nn
\begin{xy}
(20,20)*+{\Res_{\Bij_k}^{\Fin} \Lan_{\Bij_{<k}}^{\Fin} \Res_{\Bij_{<k}}^{\Fin} V}="a";
(20,0)*+{\im}="b";
(70,20)*+{\Res_{\Bij_k}^{\Fin} V}="c";
(70,0)*+{\Res_{\Bij_k}^{\Fin} W}="d";
(100,20)*+{H_0V[k]}="e";
(100,0)*+{H_0W[k]}="f";
(130,20)*+{0\phantom{.}}="g";
(130,0)*+{0.}="h";
(-10,0)*+{0}="z";
{\ar "a"; "b"};
{\ar "c"; "d"};
{\ar "e"; "f"};
{\ar "a"; "c"};
{\ar^h "c"; "e"};
{\ar "e"; "g"};
{\ar "b"; "d"};
{\ar_h "d"; "f"};
{\ar "f"; "h"};
{\ar "z"; "b"};
\end{xy}
\ee
The left vertical arrow is still a surjection since it is a composition of two surjections.  It follows by the snake lemma that the middle vertical arrow is a surjection, as required.
\end{proof}
\subsection{Homology of Schur projectives}
\begin{prop} \label{schurhomology}
If $\mathbb{P}$ is a Schur projective of degree $k$ then
\be \nn
H_0 \mathbb{P}[l] = 0
\ee
for every $l > k$.
\end{prop}
\begin{proof}
Since $l > k$, we may write
\be \nn
\mathbb{P} = \Lan_{\Bij_{< l}}^{\Fin} W
\ee
for some functor $W : \Bij_{< l} \longrightarrow \Vect$.  Consequently,
\bea \nn
H_0 \mathbb{P}[l] & = & \coker \left( \Lan_{\Bij_{< l}}^{\Fin} \Res_{\Bij_{< l}}^{\Fin} \mathbb{P} \longrightarrow \mathbb{P} \right)[l] \\ \nn
& = & \coker \left( \Lan_{\Bij_{<l}}^{\Fin} \Res_{\Bij_{< l}}^{\Fin}\Lan_{\Bij_{<l}}^{\Fin} W \longrightarrow \Lan_{\Bij_{< l}}^{\Fin} W \right)[l] \\ \nn
& = & 0,
\eea
where this cokernel vanishes since the map is split epic by the counit-unit equations.
\end{proof}
\begin{prop} \label{schurhomologyindegree}
Given a Schur projective $\mathbb{P}$ of (pure) degree $k$ written $\mathbb{P} = \Lan_{\Bij_{k}}^{\Fin} W$,
the zeroth homology of $\mathbb{P}$ satisfies
\be \nn
H_0 \mathbb{P}[k] \simeq W,
\ee
an isomorphism of $\Bij_k$-representations.
\end{prop}
\begin{proof}
To begin, let's take $W = \mathbb{Q} \Bij_k$, the regular representation.  For any representation $V$,
\be \nn
\Vect^{\Fin}(\Lan^{\Fin}_{\Bij_k} \mathbb{Q} \Bij_k, V) \simeq \Vect^{\Bij_k}(\mathbb{Q} \Bij_k, \Res^{\Fin}_{\Bij_k} V) \simeq V[k],
\ee
and so
\be \nn
\Lan^{\Fin}_{\Bij_k} \mathbb{Q} \Bij_k \simeq \tensorpower^k
\ee
since they represent the same functor.  By the definition of homology, we get an exact sequence
\be \nn
\left(\raisebox{1pt}{\scalebox{.7}{$\Vect^{\Fin}\left(\tensorpower^{< k}, \tensorpower^k \right)$}} \otimes \tensorpower^{< k} \right) [k] \longrightarrow \tensorpower^k[k] \longrightarrow H_0\tensorpower^k[k] \longrightarrow 0.
\ee
Expanding, the sequence becomes
\be \nn
\bigoplus_{n < k} \QFin([k],[n]) \otimes  \QFin([n],[k])  \longrightarrow \QFin([k],[k]) \longrightarrow H_0\tensorpower^k[k] \longrightarrow 0.
\ee
The first map is given by composition, so its image is the span of all maps factoring through a set of cardinality less than $k$; this leaves only bijections for the cokernel.  It follows that
\be \nn
H_0 \Lan^{\Fin}_{\Bij_k} \mathbb{Q} \Bij_k \simeq \mathbb{Q} \Bij_k.
\ee
Since $H_0$ and $\Lan$ are additive, we obtain the conclusion for any summand of $\mathbb{Q} \Bij_k$ as well.  But any $W$ is a sum of such summands, and the claim is proved.
\end{proof}
\begin{lem} \label{goodcover}
Given a representation $V : \Fin \longrightarrow \Vect $ and an injection $i$ splitting the homology map
\be \nn
\begin{xy}
(0,10)*+{H_0V}="a"; (35,0)*+{\mathbb{R} V}="b"; (70,10)*+{H_0V,}="c"; 
{\ar_i "a"; "b"};
{\ar_h "b"; "c"};
{\ar^{1} "a"; "c"};
\end{xy}
\ee
the composite
\be \nn
\begin{xy}
(0,0)*+{\mathbb{S} H_0V}="a"; (30,0)*+{\mathbb{S} \mathbb{R} V}="b"; (60,0)*+{V}="c"; 
{\ar^{\mathbb{S} i} "a"; "b"};
{\ar^{\epsilon_{\Bij}^{\Fin}V} "b"; "c"};
\end{xy}
\ee
is a surjection (the second map is the counit of the $\mathbb{S} \dashv \mathbb{R}$ adjunction).  Furthermore, if $V$ is finitely generated of degree $k$ then the induced map
\be \nn
\begin{xy}
(0,0)*+{H_0\mathbb{S}H_0V[k]}="a"; (40,0)*+{H_0V[k]}="c"; 
{\ar "a"; "c"};
\end{xy}
\ee
is an isomorphism.
\end{lem}
\begin{proof}
Build the diagram
\be \nn
\begin{xy}
(0,40)*+{H_0 V}="a"; (50,40)*+{\mathbb{R}V}="b";
(0,20)*+{\mathbb{R}\mathbb{S}H_0 V}="d";
(50,20)*+{\mathbb{R}\mathbb{S}\mathbb{R}V}="e";
(90,20)*+{\mathbb{R}V}="f";
(0,0)*+{H_0\mathbb{S}H_0 V}="g";
(50,0)*+{H_0\mathbb{S}\mathbb{R}V}="h";
(90,0)*+{H_0V.}="i";
{\ar^{i} "a"; "b"};
{\ar^{\mathbb{R} \mathbb{S} i} "d"; "e"};
{\ar^{\mathbb{R} \epsilon V} "e"; "f"};
{\ar^{H_0 \mathbb{S} i} "g"; "h"};
{\ar^{H_0 \epsilon V} "h"; "i"};
{\ar_{\eta H_0 V} "a"; "d"};
{\ar_{h \mathbb{S} H_0 V} "d"; "g"};
{\ar_{\eta \mathbb{R} V} "b"; "e"};
{\ar_{h \mathbb{S} \mathbb{R} V} "e"; "h"};
{\ar^{h} "f"; "i"};
{\ar^1 "b"; "f"};
\end{xy}
\ee
The left three horizontal maps are induced by $i$; the right two are built from the counit $\epsilon_{\Bij}^{\Fin}$; the bottom two squares commute by naturality of the homology map $h$; the upper right triangle is one of the counit-unit identities; and the upper left square commutes by naturality of the unit $\eta_{\Bij}^{\Fin}$.  The composite along the upper right gives an identity map, so the composite along the bottom
\be \nn
(H_0 \mathbb{S} i ) ( H_0 \epsilon_{\Bij}^{\Fin} V) = H_0 \left((\mathbb{S} i )(\epsilon_{\Bij}^{\Fin} V)\right)
\ee
is epic.  As required, Proposition \ref{homreflectssurjections} gives that $(\mathbb{S} i )(\epsilon_{\Bij}^{\Fin} V)$ is also epic. \\ \\
\noindent
In the case that $V$ is finitely generated of degree $k$, we show that the surjection
\be \nn
H_0 \mathbb{S} H_0 V[k] \longrightarrow  H_0 V[k] \longrightarrow 0;
\ee
is actually an isomorphism.  Since the category of representations of $\Bij$ is semisimple, we may write $H_0V[k]$ as a direct sum of irreducible representations of various symmetric groups. The functors $\Lan_{\Bij}^{\Fin}$ and $H_0$ are additive, so the result follows from Proposition \ref{schurhomology} and Proposition \ref{schurhomologyindegree}.
\end{proof}
\noindent
Armed with Lemma \ref{goodcover} and Propositions \ref{schurhomology} and \ref{schurhomologyindegree}, we observe an alternative definition of the degree of a representation $V \in \fg$.
\begin{obs}[Degree in terms of homology]\label{degreebyhomology}
A finitely generated $\Fin$-representation $V$ has degree $k$ if and only if its zeroth homology is supported on sets of cardinality at most $k$.
\end{obs}
\begin{prop}[Pruning homology]\label{prune}
A subrepresentation $K \subseteq \mathbb{P}$ of a Schur projective of degree $k$ satisfies
\bea \nn
\vdots \hspace{.2in} & \simeq & 0 \\ \nn
H_0K[k+3] & \simeq & 0 \\ \nn
H_0K[k+2] & \simeq & 0 \\
H_0K[k+1] & \simeq & H_0K[k+1]\varepsilon. \label{notkilled}
\eea
\end{prop}
\begin{proof}
Any such $\mathbb{P}$ is a direct summand of a direct sum of copies of $\tensorpower^{\leq k}$, so $\mathbb{P}$ squishes any $\tensorpower^n$ through $\tensorpower^{\leq k+1}$ by the Upper Squishing Lemma \ref{uppersquisher}.  All the isomorphisms except (\ref{notkilled}) can now be deduced from Lemma \ref{killhomology}.  The Lower Squishing Lemma \ref{lowersquisher} gives the last isomorphism in a similar fashion.
\end{proof}
\begin{prop} \label{shapelem}
Let $\mathbb{P}$ be a Schur projective of degree $k$ and suppose we have an injection
\be \nn
K \hookrightarrow \mathbb{P}.
\ee
The induced map
\be \nn
H_0K[k](1-\varepsilon) \longrightarrow H_0\mathbb{P}[k](1-\varepsilon)
\ee
is also an injection.
\end{prop}
\begin{proof}
Build the diagram
\be \nn
\begin{xy}
(20,55)*+{\UP(\Theta^k, \mathbb{P})}="a";
(20,40)*+{\UP(\Theta^k, \sk_{<k} \mathbb{P}) + \UP(\Theta^k, K)}="b";
(0,20)*+{\UP(\Theta^k, \sk_{<k} \mathbb{P})}="c"; (40,20)*+{\UP(\Theta^k, K)}="g";
(20,0)*+{\UP(\Theta^k, (\sk_{<k} \mathbb{P}) \cap K).}="h";
{\ar "b"; "a"};
{\ar "c"; "b"};
{\ar "h"; "g"};
{\ar "g"; "b"};
{\ar "h"; "c"};
\end{xy}
\ee
Each arrow is either a natural inclusion or an inclusion by exactness of $\UP(\Theta^k,-)$.   Exactness also gives that the diamond is cartesian.  The second isomorphism theorem gives an injection
\be \nn
\frac{\UP(\Theta^k, K)}{ \UP(\Theta^k, (\sk_{<k} \mathbb{P}) \cap K)} \simeq \frac{\UP(\Theta^k, \sk_{<k} \mathbb{P}) + \UP(\Theta^k, K)}{\UP(\Theta^k, \sk_{<k} \mathbb{P})} \hookrightarrow \frac{\UP(\Theta^k, \mathbb{P}) }{\UP(\Theta^k, \sk_{<k} \mathbb{P})}.
\ee
The Lower Squishing Lemma \ref{lowersquisher} gives that $\tensorpower^{<k}$ squishes $\Theta^k$ through $\tensorpower^{<k}$, and so we may apply Lemma \ref{intersectlem} to the inclusion $K \hookrightarrow \mathbb{P}$, yielding an isomorphism
\be \nn
\UP(\Theta^k, (\sk_{<k} \mathbb{P}) \cap K) = \UP(\Theta^k, \sk_{<k}  K),
\ee
which is to say
\be \nn
((\sk_{<k} \mathbb{P}) \cap K)[k](1-\varepsilon) = (\sk_{<k}  K)[k](1-\varepsilon).
\ee
By definition of zeroth homology,
\be \nn
H_0 K[k](1-\varepsilon) \simeq \frac{\UP(\Theta^k, K)}{ \UP(\Theta^k, \sk_{<k} K)} \simeq \frac{\UP(\Theta^k, K)}{ \UP(\Theta^k, (\sk_{<k} \mathbb{P}) \cap K)} \hookrightarrow \frac{\UP(\Theta^k, \mathbb{P}) }{\UP(\Theta^k, \sk_{<k} \mathbb{P})} \simeq H_0 \mathbb{P}[k](1-\varepsilon),
\ee
and the claim is proved.
\end{proof}
\begin{lem}\label{sgnlem}
Suppose the symmetric group $\Bij_k$ acts transitively on a set $X$, and so acts on the free vector space $\mathbb{Q} X$.  If the stabilizer of a point in $X$ lies inside the alternating group $\mathcal{A}_k \subset \Bij_k$, then the representation $\mathbb{Q} X$ contains exactly one copy of the sign representation.  Otherwise, $\mathbb{Q} X$ contains no copies of the sign representation.
\end{lem}
\begin{proof}
Write $X \simeq \Bij_k/\mathcal{H}$, and let $W_{\varepsilon}$ and $\mathbb{Q}$ denote the sign representation of $\Bij_k$ and trivial representation of $\mathcal{H}$ respectively.  By Frobenius Reciprocity,
\bea \nn
\dim \Vect^{\Bij_k}( W_{\varepsilon}, \; \mathbb{Q} X) & = & \dim \Vect^{\Bij_k}( W_{\varepsilon}, \; \Ind_{\mathcal{H}}^{\Bij_k} \mathbb{Q}) \\ \nn
& = & \dim \Vect^{\mathcal{H}}( \Res_{\mathcal{H}}^{\Bij_k}W_{\varepsilon}, \; \mathbb{Q}),
\eea
which is $1$ if $\mathcal{H} \subseteq \mathcal{A}_k$, and $0$ otherwise.
\end{proof}
\begin{prop} \label{dimcount}
We compute
\bea \nn
\dim \UP(\Lambda^k, \tensorpower^n) &=& \stirling{n}{k-1} +  \stirling{n}{k} \\ \nn
\dim \UP(D_k, \tensorpower^n ) &=& \stirling{n}{k-1},
\eea
where the symbol $\stirling{n}{k}$ denotes the number of partitions of an $n$-element set into exactly $k$ subsets (a Stirling number of the second kind; see \cite{EnumCombo1}).
\end{prop}
\begin{proof}
By the universal property of $\Lambda^k$,
\be \nn
\dim \UP(\Lambda^k, \tensorpower^n) = \mbox{multiplicity of sign rep in the $\Bij_k$ representation $\mathbb{Q} \Fin([n],[k])$}
\ee
Applying Lemma \ref{sgnlem}, this multiplicity may be computed by studying the action of $\Bij_k$ on the finite set $\Fin([n],[k])$: it equals the number orbits whose stabilizers lie inside the alternating group $\mathcal{A}_k$.  The right stabilizer of a function to $[k]$ is the full symmetric group on the complement of the image.  In order for the stabilizer to stay in the alternating group, the complement of the image must have cardinality at most $1$.  In other words, the multiplicity equals the number of orbits consisting of functions which are either surjective, or else miss at most one point.  The orbits of $\Fin([n],[k])$ coincide with set partitions of $[n]$, and the first equality follows.  The second equality follows from the first since 
\be \nn
\UP(D_k, \tensorpower^n ) = \ker \left( \UP(\Lambda^k, \tensorpower^n) \overset{d^*}{\longrightarrow}  \UP(\Lambda^{k+1}, \tensorpower^n)  \right),
\ee
picks the non-surjective functions.
\end{proof}
\begin{prop}
Let $\mathbb{P}$ be a Schur projective of degree $k$ and $f:\Lambda^k \longrightarrow \mathbb{P}$ a map of representations.  The following conditions are equivalent:
\begin{itemize}
\item The map $f$ fails to be an injection;
\item The map $f$ factors through the canonical surjection $\Lambda^k \longrightarrow D_k$;
\item The induced map $H_0f[k]$ vanishes.
\end{itemize}
\end{prop}
\begin{proof}
The first two conditions are equivalent since $\Lambda^k$ has exactly one nontrivial subobject, and its quotient by that subobject is isomorphic to $D_k$. To prove the second two conditions are equivalent, build the chain complex
\be \nn
0 \longrightarrow \UP(\Lambda^k, D_k, \mathbb{P}) \longrightarrow \UP(\Lambda^k, \mathbb{P}) \longrightarrow \Vect(H_0\Lambda^k[k], H_0\mathbb{P}[k]) \longrightarrow 0,
\ee
where the first map is the natural inclusion, and the second map sends a map to its induced map on homology; the maps compose to zero since $H_0 D_k [k] = 0$.  Since $\dim \UP(\Lambda^k, D_k) = 1$, it suffices to show that this sequence is actually exact.  Every $\mathbb{P}$ is a direct sum of direct summands of functors $\tensorpower^n$ for various $n \leq k$, so it suffices to prove exactness of the chain complex
\be \nn
0 \longrightarrow \UP(\Lambda^k, D_k, \tensorpower^n) \longrightarrow \UP(\Lambda^k, \tensorpower^n) \longrightarrow \Vect(H_0\Lambda^k[k], H_0\tensorpower^n[k]) \longrightarrow 0
\ee
for every $n \leq k$.  Proposition \ref{dimcount} gives the dimensions of the first two vector spaces:
\bea \nn
\dim \UP(\Lambda^k, D_k, \tensorpower^n) &=& \stirling{n}{k-1} \\ \nn
\dim \UP(\Lambda^k, \tensorpower^n) &=& \stirling{n}{k-1} + \stirling{n}{k}.
\eea
In the case $n < k$, we see that $\stirling{n}{k} = 0$ and $H_0\tensorpower^n[k] = 0$, so the sequence is exact.  When $n = k$, we see $\dim \Vect(H_0\Lambda^k[k], H_0\tensorpower^k) = 1$ since this is the multiplicity of the sign representation in the regular representation of $\Bij_k$.  On the other hand, $\stirling{k}{k} = 1$, and so the sequence is still exact.
\end{proof}
\begin{cor} \label{wedgefactor}
Given a vector space $Z$ and a map $Z \otimes \Lambda^{k} \longrightarrow \mathbb{P}$ to a Schur projective of degree $k$, the induced map
\be \nn
H_0 \left( Z \otimes \Lambda^{k} \right) [k] \longrightarrow H_0\mathbb{P}[k]
\ee
is zero if and only if the original map factors through the canonical surjection
\be \nn
Z \otimes \Lambda^{k} \longrightarrow Z \otimes D_{k} \longrightarrow \mathbb{P}.
\ee
\end{cor}
\begin{prop}
Let $\mathbb{P}$ be a Schur projective of degree $k$.  Any injection  $\Lambda^k \hookrightarrow \mathbb{P} $ is split.
\end{prop}
\begin{proof}
Let $\lambda$ be a partition of size at most $k$.  Observe from Proposition \ref{dimcount} that all maps $\Lambda^k \longrightarrow \mathbb{P}_{\lambda}$  factor through $D_k$ unless $\lambda$ is a column of height $k$.  Writing $\mathbb{P}$ as a sum of various isotype projectives, we see that the existence of an injection
\be \nn
\Lambda^k \hookrightarrow \bigoplus_{\lambda} \mathbb{P}_{\lambda}
\ee
implies that one of the summands is actually a copy of $\Lambda^k$ and that the injection restricted to that summand is a non-zero multiple of the identity map.  It follows that the injection is split by projection onto this factor.
\end{proof}
\noindent
Before embarking on the proof of the main theorem \ref{mainB}, we provide the proofs of the minor theorems.
\subsection{Proofs of minor theorems} \label{minorproofs}
\subsubsection{Proof of Theorem \ref{upvequalsfinrep} equating uniformly presented vector spaces and finitely generated $\Fin$-representations} \label{pupvequalsfinrep}
\begin{proof}
Any imrep is finitely generated, so any uniformly presented vector space---defined as a quotient of an imrep---is as well.  Conversely, a finitely generated $\Fin$-representation $V$ is a quotient of some $\langle Y \rangle$ for $Y \in \QFin$; it remains to show that any subrepresentation of $\langle Y \rangle$ is an imrep.  By the Upper Squishing Lemma \ref{uppersquisher}, any subrepresentation $K \subseteq \langle Y \rangle$ is generated in some finite degree $k$.  Any basis for $K[0] \oplus K[1] \oplus \cdots \oplus K[k]$ supplies the required imrep.
\end{proof}
\subsubsection{Proof of Theorem \ref{fgisfl}: finitely generated implies finite length} \label{pfgisfl}
\begin{proof}
Let $V$ be a finitely generated $\Fin$-representation of degree $k$.  It follows from the upper squishing lemma that $V$ squishes any $\tensorpower^l$ through $\tensorpower^{k+1}$.  The same holds for any subobjects of $V$.  Let
\be \label{noethersubthings}
0 \subseteq V_1 \subseteq V_2 \subseteq V_3 \subseteq \cdots 
\ee
be an increasing chain of subobjects.  Applying the restriction functor $\Res_{\Fin_{\leq k}}^{\Fin}$ forces the sequence to stabilize past some number $N$:
\be \nn
\Res_{\Fin_{\leq k}}^{\Fin} V_N =  \Res_{\Fin_{\leq k}}^{\Fin} V_{N+1} =  \Res_{\Fin_{\leq k}}^{\Fin} V_{N+2} = \cdots
\ee
For any $l \leq k$ and representation $W$, the vector space $H_0W[l]$ depends only on $ \Res_{\Fin_{\leq k}}^{\Fin} W$.  It follows that
\be \nn
H_0 V_N[l] = H_0 V_{N+1}[l] = H_0 V_{N+2}[l] = \cdots
\ee
Stabilization for $l>k$ is trivial since these homology groups vanish by Proposition \ref{killhomology}. We see that the inclusions past the $N^{th}$ step induce isomorphisms on all of zeroth homology.  Proposition \ref{homreflectssurjections} tells us that the inclusions (\ref{noethersubthings}) are eventually surjective and so the sequence stabilizes.  A similar argument proves the descending chain condition.
\end{proof}

\subsubsection{Proof of Theorem \ref{simpleclassification} classifying simple representations} \label{psimpleclassification}
\begin{proof}
We begin by showing the $C_{\lambda}$ are simple representations.  The Lower Squishing Lemma \ref{lowersquisher} says that $\Theta^k$ squishes any $\tensorpower^n$ through $\tensorpower^{\leq k}$.  Since each $C_{\lambda}$ listed in the theorem is a quotient of some $\Theta^k$, these representations also squish any $\tensorpower^n$ through $\tensorpower^{\leq k}$, so it suffices to look for subrepresentations generated by vectors living in $\sk_{\leq k} C_{\lambda}[k]$.  But $\Res^{\Fin}_{\Bij_{\leq k}} C_{\lambda}$ is a simple $\Bij_{\leq k}$ module, so any such subrepresentation is $0$ or $C_{\lambda}$. \\ \\
\noindent
Similarly, the representation $D_k$ is simple: it squishes any $\tensorpower^n$ through $\tensorpower^{\leq k+1},$ and $\Res^{\Fin}_{\Bij_{\leq k+1}} D_k$ is itself simple. \\ \\
\noindent
To show that this list of simples is complete, it suffices to show that a representation $V$ is either zero or else contains a subrepresentation isomorphic to some irreducible $C_{\lambda}$ or $D_k$.  Given a non-zero representation $V$, let $k$ be the smallest natural number so that $V[k]$ is non-zero.  Any irreducible summand $W$ of the $\Bij_k$-representation $\Res^{\Fin}_{\Bij_k}V$ provides a non-zero map from $\mathbb{S}W$ sending  $\sk_{<k}\mathbb{S}W$ to zero; in other words, we have a non-zero map from $C_{\lambda}$.  In the case where $\lambda$ is not a column, this map must be an injection since $C_{\lambda}$ is irreducible.  If $\lambda$ is a column, then $C_{\lambda} \simeq \Lambda^k$ and so we either have an injection, and hence a subrepresentation isomorphic to $\Lambda^k$ (which then has $D_{k+1}$ as a subrepresentation) or else we have a non-injection, which must have image isomorphic to $D_k$.
\end{proof}
\subsubsection{Proof of Corollary \ref{polydimgivesfg} on $\Fin$-representations with polynomial growth} \label{ppolydimgivesfg}
By the classification of simple objects, any simple composition factor of an $\Fin$-representation $V$ with growth bounded by a polynomial of degree $d$ must be one of the $C_{\lambda}$ with $\lambda$ a partition with at most $d$ boxes, or else one of the objects $D_0$, $D_1$, \ldots, $D_{d+1}$.  All of these representations squish any $\tensorpower^n$ through $\tensorpower^{\leq d + 1}$, so $V$---an extension of such representations---also squishes any $\tensorpower^n$ through $\tensorpower^{\leq d + 1}$.  It follows that $\sk_{d+1} V \simeq V$.  The more precise version follows by a similar argument using the Lower Squishing Lemma \ref{lowersquisher}.
\subsubsection{Proof of Corollary \ref{tracestats} on polynomiality of characters} \label{ptracestats}
\begin{proof}
It suffices to check the claim on a basis for $K$-theory.  Symmetric function theory (see \cite{EnumCombo2}) gives the result for all Schur projectives $\mathbb{P}_{\lambda}$, using power sums.  The only remaining representation to check is $D_0$, but that functor is supported on the empty set.
\end{proof}
\subsubsection{Proof of dimension formula \ref{clambdadim} for the simple representations $C_{\lambda}$} \label{pclambdadim}
\begin{proof}
Since $C_{\lambda} [0] \simeq 0$, $ C_{\lambda}$ has no composition factor isomorphic to $D_0$.  It follows that the dimension of $ C_{\lambda} [n]$ is exactly polynomial in $n$, necessarily polynomial of degree at most $k$.  But there is only one polynomial of degree $k$ vanishing at $0, 1, 2, \ldots, k-1$ and taking the value $(\dim \Sp_{\lambda})$ at $k$.
\end{proof}

\subsection{The main result: refined statement}
\begin{thm}[Version B] \label{mainB}
Any finitely generated $\Fin$-representation $V$ of degree $k$ has a finite resolution
\be \nn
0 \longrightarrow \mathbb{P}_k \oplus \mathbb{D}_k \longrightarrow \cdots \longrightarrow \mathbb{P}_2 \oplus \mathbb{D}_2 \longrightarrow  \mathbb{P}_1 \oplus \mathbb{D}_1 \longrightarrow  \mathbb{P}_0 \longrightarrow V \longrightarrow 0
\ee
where each $\mathbb{P}_i$ is a Schur projective of degree $k-i$ and each $\mathbb{D}_i$ is a direct sum of copies of the objects $D_{k-i+1}$ and $D_{k-i+2}$.
\end{thm}
\subsection{The main result: proof}
\begin{proof}
For the purpose of induction, suppose we have the result for all uniformly presented vector spaces of degrees up to $k$.  Given a uniformly presented vector space $V$ of degree $k + 1$ we may\footnote{Here we use that $\Vect^{\Bij}$ is semisimple and so every surjection is (non-canonically) split.} use Lemma \ref{goodcover} to produce a (non-canonical) surjection
\be \nn
\mathbb{S} H_0 V \longrightarrow V \longrightarrow 0
\ee
so that the induced map
\be \label{finaliso}
0 \longrightarrow H_0 \mathbb{S} H_0 V[k+1] \longrightarrow H_0V[k+1] \longrightarrow 0 
\ee
is an isomorphism.
Let $K$ denote the kernel
\be \label{kdef}
0 \longrightarrow K \longrightarrow \mathbb{S} H_0 V \longrightarrow V \longrightarrow 0.
\ee
Applying $H_0$ to this short exact sequence evaluated at the object $[k+1]$ gives a final isomorphism by (\ref{finaliso}):
\be \label{lesksv}
 H_0K[k+1] \longrightarrow H_0 \mathbb{S} H_0 V[k+1] \overset{\sim}{\longrightarrow} H_0V[k+1] \longrightarrow 0.
\ee
It follows that the penultimate map is a zero map.  Projecting away from the $S_{k+1}$ anti-invariants, the exact sequence becomes
\be \nn
H_0K[k+1](1-\varepsilon) \overset{0}{\longrightarrow} H_0 \mathbb{S} H_0 V[k+1](1-\varepsilon) \overset{\sim}{\longrightarrow} H_0V[k+1](1-\varepsilon) \longrightarrow 0.
\ee
By Proposition \ref{shapelem} we know that the map marked zero is an injection.  It follows that $H_0K[k+1](1-\varepsilon) = 0$.  This fact, along with the consequences of pruning \ref{prune} are tabulated here:
\bea \nn
\vdots \hspace{.2in} & \simeq & 0 \\ \nn
H_0K[k+4] & \simeq & 0 \\ \nn
H_0K[k+3] & \simeq & 0 \\
H_0K[k+2] & \simeq & H_0K[k+2]\varepsilon \label{k2homology} \\
H_0K[k+1] & \simeq & H_0K[k+1]\varepsilon. \label{k1homology}
\eea
In what follows, set $l = k+2$ or $l=k+1$.  The argument proceeds in parallel in these two degrees.    Picking representatives for homology gives
\be \nn
H_0K[l] \otimes \Lambda^{l} \longrightarrow K,
\ee
which may be extended by the canonical inclusion $K \hookrightarrow \mathbb{S}H_0 V$
\be \nn
H_0K[l] \otimes \Lambda^{l} \longrightarrow K\longrightarrow \mathbb{S}H_0 V.
\ee
Applying the functor $H_0$ and evaluating at the set $[l]$,
\be \nn
H_0(H_0K[l] \otimes \Lambda^{l})[l] \longrightarrow H_0K[l]\longrightarrow H_0\mathbb{S}H_0 V[l].
\ee
For either value of $l$, this composite is zero: for $l=k+2$, $H_0\mathbb{S}H_0 V[l]$ is zero by Proposition \ref{schurhomology}; for $l=k+1$, the second map is zero since it precedes an isomorphism in the exact sequence (\ref{lesksv}).  By Proposition \ref{wedgefactor}, these two composites factor:
\be \nn
\begin{xy}
(0,20)*+{H_0K[l] \hspace{-3pt} \otimes \hspace{-3pt} \Lambda^{l}}="a"; (45,20)*+{K}="b"; 
(0,0)*+{H_0K[l] \hspace{-3pt} \otimes \hspace{-3pt} D_{l}}="c"; (45,0)*+{\mathbb{S}H_0V.}="d"; 
{\ar "a"; "b"};
{\ar "b"; "d"};
{\ar "a"; "c"};
{\ar "c"; "d"};
\end{xy}
\ee
The left map is epic and the right map is monic, so we get a unique lift
\be \label{liftsquare}
\begin{xy}
(0,20)*+{H_0K[l] \hspace{-3pt} \otimes \hspace{-3pt} \Lambda^{l}}="a"; (45,20)*+{K}="b"; 
(0,0)*+{H_0K[l] \hspace{-3pt} \otimes \hspace{-3pt} D_{l}}="c"; (45,0)*+{\mathbb{S}H_0V.}="d"; 
{\ar "a"; "b"};
{\ar "b"; "d"};
{\ar "a"; "c"};
{\ar "c"; "d"};
{\ar "c"; "b"};
\end{xy}
\ee
Taking homology and evaluating at $[l]$,
\be \nn
\begin{xy}
(0,20)*+{H_0(\scalebox{1}{$H_0K[l] \hspace{-3pt} \otimes \hspace{-3pt} \Lambda^{l}$})[l]}="a"; (45,20)*+{H_0(\scalebox{1}{$K$})[l]}="b"; 
(0,0)*+{H_0(\scalebox{1}{$H_0K[l] \hspace{-3pt} \otimes \hspace{-3pt} D_{l}$})[l]}="c"; (45,0)*+{H_0(\scalebox{1}{$\mathbb{S}H_0V$})[l].}="d"; 
{\ar "a"; "b"};
{\ar "b"; "d"};
{\ar "a"; "c"};
{\ar "c"; "d"};
{\ar "c"; "b"};
\end{xy}
\ee
By Lemma \ref{goodcover} and the previous computations (\ref{k2homology}) and (\ref{k1homology}), the top map is an isomorphism.  The left map is still surjective by the right exactness of zeroth homology.  It follows that the diagonal map is also an isomorphism:
\be \label{diagiso}
H_0(\scalebox{1}{$H_0K[l] \hspace{-3pt} \otimes \hspace{-3pt} D_{l}$})[l] \overset{\sim}{\longrightarrow} H_0(\scalebox{1}{$K$})[l].
\ee
\\ 
\\ \noindent
We need to establish that the diagonal lifting map from (\ref{liftsquare})
\bea \label{lifters}
H_0K[l] \otimes D_{l} \longrightarrow K
\eea
is an injection\footnote{At this stage, there are several fast ways to finish proof of the main theorem without showing this map is an injection.   We favor a style that gives a more practical algorithm, however.}.  Define the kernel
\be \nn
0 \longrightarrow N \longrightarrow H_0K[l] \otimes D_{l} \longrightarrow K.
\ee
Choosing a basis for $H_0K[l]$ we see that $H_0K[l] \otimes D_{l}$ has exactly $\dim H_0K[l]$ composition factors all of which are isomorphic to the simple object $D_{l}$.  Since $\Ext^1(D_l, D_l)=0$, the inclusion of $N$ is split and so the identity map on $N$ factors
\be \nn
N \longrightarrow H_0K[l] \otimes D_{l} \longrightarrow N.
\ee 
Applying the functor $H_0$ and evaluating at $[l]$, we have a corresponding factorization of the identity map on $H_0N[l]$:
\be \nn
H_0N[l] \longrightarrow H_0(H_0K[l] \otimes D_{l})[l] \longrightarrow H_0N[l].
\ee
It follows that this first map is monic.  Let $M$ be the cokernel of the natural inclusion $N \longrightarrow H_0K[l] \otimes D_{l}$:
\be \label{nmles}
0 \longrightarrow N \longrightarrow H_0K[l] \otimes D_{l} \longrightarrow M \longrightarrow 0.
\ee
By (\ref{diagiso}), we know that the composite
\be \nn
H_0(H_0K[l] \otimes D_{l})[l] \longrightarrow H_0N[l] \longrightarrow H_0K[l]
\ee
is an isomorphism.  It follows that this first map is monic as well.  We get that the exact sequence associated to (\ref{nmles}) ends with an isomorphism
\be \nn
H_0N[l] \longrightarrow H_0(H_0K[l] \otimes D_{l})[l] \overset{\sim}{\longrightarrow} H_0M[l] \longrightarrow 0.
\ee
The map $H_0N[l] \longrightarrow H_0(H_0K[l] \otimes D_{l})[l]$ (which is now seen to be a zero map) was earlier found to be monic.  It follows that
\be \nn
H_0N[l] = 0.
\ee
But $H_0D_l[l] \simeq \mathbb{Q}$ and so
\be \nn
N = 0.
\ee \\
\\ \noindent
We have shown that the maps (\ref{lifters}) are injective.  In fact, their direct sum
\be \nn
\left(
\begin{array}{c}
H_0K[k+2] \otimes D_{k+2}\\
\oplus \\
H_0K[k+1] \otimes D_{k+1}\\
\end{array}
\right) \longrightarrow K
\ee
remains an injection since their images can share no composition factor.
Define $V'$ to fit in the short exact sequence
\be \nn
0 \longrightarrow \left(
\begin{array}{c}
H_0K[k+2] \otimes D_{k+2}\\
\oplus \\
H_0K[k+1] \otimes D_{k+1}\\
\end{array}
\right) \longrightarrow K \longrightarrow V' \longrightarrow 0.
\ee
The associated exact sequences
\be \nn
H_0\left(
\begin{array}{c}
H_0K[k+2] \otimes D_{k+2}\\
\oplus \\
H_0K[k+1] \otimes D_{k+1}\\
\end{array}
\right)[k+2] {\longrightarrow} H_0K[k+2] \longrightarrow H_0V'[k+2] \longrightarrow 0
\ee
\vspace{5pt}
\be \nn
H_0\left(
\begin{array}{c}
H_0K[k+2] \otimes D_{k+2}\\
\oplus \\
H_0K[k+1] \otimes D_{k+1}\\
\end{array}
\right)[k+1] {\longrightarrow} H_0K[k+1] \longrightarrow H_0V'[k+1] \longrightarrow 0
\ee
have penultimate surjections by (\ref{diagiso}).  Consequently,
\bea \nn
H_0V'[k+2] = 0\phantom{.} \\ \nn
H_0V'[k+1] = 0.
\eea
It follows that the zeroth homology of $V'$ is supported in degrees $k$ and below and so $V'$ has degree $k$.  The inductive hypothesis supplies a resolution of the required form
\be \nn
\cdots \longrightarrow R_2 \longrightarrow R_1 \longrightarrow R_0 \longrightarrow V' \longrightarrow 0.
\ee
Build the diagram
\be \nn
\begin{xy}
(0,60)*+{0}="y3";
(0,40)*+{0}="y2";
(0,20)*+{0}="y1";
(0,0)*+{0}="y0";
(30,80)*+{\vdots}="a4";
(30,60)*+{0}="a3";
(30,40)*+{0}="a2";
(30,20)*+{\scalebox{.65}{$\left(
\begin{array}{c}
H_0K[k+2] \hspace{-2.5pt} \otimes \hspace{-2.5pt} D_{k+2}\\
\oplus \\
H_0K[k+1] \hspace{-2.5pt} \otimes \hspace{-2.5pt} D_{k+1}\\
\end{array}\right)$}}="a1";
(30,0)*+{\scalebox{.65}{$\left(
\begin{array}{c}
H_0K[k+2] \hspace{-2.5pt} \otimes \hspace{-2.5pt} D_{k+2}\\
\oplus \\
H_0K[k+1] \hspace{-2.5pt} \otimes \hspace{-2.5pt} D_{k+1}\\
\end{array}\right)$}}="a0";
(30,-20)*+{0}="am";
(70,80)*+{\vdots}="b4";
(70,60)*+{R_2}="b3";
(70,40)*+{R_1}="b2";
(70,20)*+{\scalebox{.65}{$\left(
\begin{array}{c}
H_0K[k+2] \hspace{-2.5pt} \otimes \hspace{-2.5pt} D_{k+2}\\
\oplus \\
H_0K[k+1] \hspace{-2.5pt} \otimes \hspace{-2.5pt} D_{k+1}\\
\end{array}\right)$}
 \oplus R_0}="b1";
(70,0)*+{K}="b0";
(70,-20)*+{0}="bm";
(110,80)*+{\vdots}="c4";
(110,60)*+{R_2}="c3";
(110,40)*+{R_1}="c2";
(110,20)*+{R_0}="c1";
(110,0)*+{V'}="c0";
(110,-20)*+{0.}="cm";
(140,60)*+{0}="z3";
(140,40)*+{0}="z2";
(140,20)*+{0}="z1";
(140,0)*+{0}="z0";
{\ar "a4"; "a3"};
{\ar "b4"; "b3"};
{\ar "c4"; "c3"};
{\ar "a3"; "a2"};
{\ar "b3"; "b2"};
{\ar "c3"; "c2"};
{\ar "a2"; "a1"};
{\ar "b2"; "b1"};
{\ar "c2"; "c1"};
{\ar "a1"; "a0"};
{\ar "b1"; "b0"};
{\ar "c1"; "c0"};
{\ar "a0"; "am"};
{\ar "b0"; "bm"};
{\ar "c0"; "cm"};
{\ar "y3"; "a3"};
{\ar "a3"; "b3"};
{\ar "b3"; "c3"};
{\ar "c3"; "z3"};
{\ar "y2"; "a2"};
{\ar "a2"; "b2"};
{\ar "b2"; "c2"};
{\ar "c2"; "z2"};
{\ar "y1"; "a1"};
{\ar "a1"; "b1"};
{\ar "b1"; "c1"};
{\ar "c1"; "z1"};
{\ar "y0"; "a0"};
{\ar "a0"; "b0"};
{\ar "b0"; "c0"};
{\ar "c0"; "z0"};
{\ar "c1"; "b0"};
\end{xy}
\ee
All the rows but the last are constructed to be split short exact sequences.  The only map that requires explanation is the diagonal map, which exists by projectivity of $R_0$.  The outer columns are exact, so the middle is as well by the long exact sequence associated to this short exact sequence of chain complexes.  Splicing the middle column to the short exact sequence (\ref{kdef}) gives
\be \nn
\cdots \longrightarrow R_2 \longrightarrow R_1 \longrightarrow \begin{array}{c}
\scalebox{.65}{$\left(
\begin{array}{c}
H_0K[k+2] \hspace{-2.5pt} \otimes \hspace{-2.5pt} D_{k+2}\\
\oplus \\
H_0K[k+1] \hspace{-2.5pt} \otimes \hspace{-2.5pt} D_{k+1}\\
\end{array}\right)$} \\
\oplus \\
R_0
\end{array} \longrightarrow \mathbb{S}H_0V \longrightarrow V \longrightarrow 0,
\ee
which is a resolution of $V$ the advertized form.
 \end{proof}
\section{Explicit resolution of $V$ from the introduction}
As an example of the algorithm outlined in the proof of Theorem \ref{mainB}, we resolve the uniformly presented vector space given in (\ref{origw}).  This example is among the smallest exhibiting all of the qualitative aspects of the main result.  Recall that
\be \nn
V = \langle x_i \otimes x_j \rangle \; \big / \; \big \langle x_i \otimes x_i + x_j \otimes x_j + x_k \otimes x_k - 3\left(x_i \otimes x_j - x_i \otimes x_k + x_j \otimes x_k \right) \big \rangle,
\ee
and that $V$ can be written as a uniformly presented vector space
\be \nn
V=
\begin{amatrixarrow}[0ex]{[2]}{[2]}
   \; \fn{12} \;
 \end{amatrixarrow} \; \big / \;
 \begin{amatrixarrow}[0ex]{[2]}{[3]}
    \; \fn{11} + \fn{22} + \fn{33} -3 \cdot \fn{12} + 3 \cdot \fn{13} -3 \cdot \fn{23}  \;
 \end{amatrixarrow}.
\ee
Since $\begin{amatrixarrow}[0ex]{[2]}{[2]}
   \; \fn{12} \;
 \end{amatrixarrow} \simeq \tensorpower^2 \simeq \mathbb{P}_{\ydiagram{1,1}} \oplus \mathbb{P}_{\ydiagram{2}}$, we may take the first step of the resolution to be
 \be \nn
 \cdots \longrightarrow \mathbb{P}_{\ydiagram{1,1}} \oplus \mathbb{P}_{\ydiagram{2}} \longrightarrow V \longrightarrow 0
 \ee
 before beginning the algorithm in earnest.  The kernel of this surjection is none other than the denominator of $V$:
 \be \nn
 I = \begin{amatrixarrow}[0ex]{[2]}{[3]}
    \; \fn{11} + \fn{22} + \fn{33} -3 \cdot \fn{12} + 3 \cdot \fn{13} -3 \cdot \fn{23}  \;
 \end{amatrixarrow}.
 \ee
 \subsection{Computing zeroth homology of an imrep}
 We must compute the zeroth homology of the imrep $I$.  Since $I$ is a quotient of $\tensorpower^3$, its zeroth homology is supported on $\Bij_{\leq 3}$.  Starting at the bottom, $I[0]$ and $I[1]$ both vanish, and so
 \bea \nn
 H_0 I [0] &=& 0 \\ \nn
 H_0 I [1] &=&  0 \\ \nn
 H_0 I [2] &\simeq& I[2].
 \eea
By definition, $I[2]$ is the composition
 \be \nn
( \fn{11} + \fn{22} + \fn{33} -3 \cdot \fn{12} + 3 \cdot \fn{13} -3 \cdot \fn{23}) \; \QFin([3], [2]).
 \ee
Taking the span of the $8$ possible substitutions
\be \nn
\langle \;
\scalebox{.7}{$ \fn{22}-\fn{11}$} \;, \; \; \; 
 \scalebox{.7}{$5 \cdot \fn{11}-3 \cdot \fn{12}-3 \cdot \fn{21}+\fn{22}$} \;, \; \; \; 
 \scalebox{.7}{$\fn{11}-\fn{22}$} \;, \; \; \; 
 \scalebox{.7}{$\fn{22}-\fn{11}$} \;, \; \; \; 
 \scalebox{.7}{$\fn{11}-3 \cdot \fn{12}-3\cdot  \fn{21}+5 \cdot \fn{22}$} \;, \; \; \; 
 \scalebox{.7}{$\fn{11}-\fn{22}$} \;, \; \; \; 
 0 \;
 \rangle,
\ee
we compute that
\be \nn
I[2] = \langle \; \fn{11}-\fn{22}\; , \; \; \; \; \fn{12}+\fn{21}-2 \cdot \fn{22} \; \rangle,
\ee
a $2$-dimensional vector space.  To compute the isotypic components under the $\Bij_2$ action, postcompose each generator with $\tau_2$ and $\varepsilon_2$:
\bea \nn
I[2]\tau_2 & = & \langle \; 0 \; , \; \; \; \; -\fn{11} + \fn{12} + \fn{21} - \fn{22} \; \rangle  \\ \nn
I[2]\varepsilon_2  & = & \langle \; \fn{11}-\fn{22} \; , \; \; \; \; \fn{11}-\fn{22} \; \rangle.
\eea
Simplifying, we have representatives for zeroth homology evaluated at $[2]$:
\bea \nn
H_0 K [2] \tau_2 & \simeq & \langle \; -\fn{11} + \fn{12} + \fn{21} - \fn{22} \; \rangle \\ \nn
H_0 K [2] \varepsilon_2 & \simeq & \langle \; \fn{11}-\fn{22} \; \rangle.
\eea
By definition, $H_0 I [3] =$
\be \nn
\scalebox{.7}{$( \fn{11} + \fn{22} + \fn{33} -3 \cdot \fn{12} + 3 \cdot \fn{13} -3 \cdot \fn{23})$} \; \QFin([3], [3]) \; \big / \; \scalebox{.7}{$( \fn{11} + \fn{22} + \fn{33} -3 \cdot \fn{12} + 3 \cdot \fn{13} -3 \cdot \fn{23})$} \; \QFin([3], [2], [3]).
\ee
Row reduction shows this space to be one-dimensional, and provides a representative
\be \nn
H_0 I [3] \simeq \langle \; \scalebox{.9}{$\frac{1}{6}$} \left( \fn{1 2} - \fn{1 3} - \fn{2 1} + \fn{2 3} + \fn{3 1} - \fn{3 2} \right) \; \rangle
\ee
which is easily seen to be anti-invariant under the action of $\Bij_3$.  Our choices of homology representatives constitute a splitting of the homology map, so we obtain the next step of the resolution
\be \nn
\cdots \longrightarrow \mathbb{P}_{\ydiagram{1,1,1}} \oplus \mathbb{P}_{\ydiagram{1,1}} \oplus \mathbb{P}_{\ydiagram{2}} \longrightarrow \mathbb{P}_{\ydiagram{1,1}} \oplus \mathbb{P}_{\ydiagram{2}} \longrightarrow V \longrightarrow 0.
\ee
\subsection{Computing the kernel of a map of Schur projectives}
Let $K$ be the kernel of the surjection onto $I$.  We know $K$ is an imrep, but we do not have its generators.  Since $K$ is the kernel of a surjection provided by Lemma \ref{goodcover}, we know $K$ is generated by its vectors in $K[0]$, $K[1]$, $K[2]$, $K[3]\varepsilon$, and $K[4]\varepsilon$.  Pick bases for the ambient vector spaces:
\bea \nn
\left(\mathbb{P}_{\ydiagram{1,1,1}} \oplus \mathbb{P}_{\ydiagram{1,1}} \oplus \mathbb{P}_{\ydiagram{2}}\right)[0] & = & \left[ \phantom{ \begin{array}{c}
\,  \\
\,  \\
\,
\end{array}} \right] \\ \nn
\left(\mathbb{P}_{\ydiagram{1,1,1}} \oplus \mathbb{P}_{\ydiagram{1,1}} \oplus \mathbb{P}_{\ydiagram{2}}\right)[1] & = & \left[  \begin{array}{c}
 0 \\
 0 \\
 \scalebox{.7}{$\fn{11}$}
\end{array}\right] \\ \nn
\left(\mathbb{P}_{\ydiagram{1,1,1}} \oplus \mathbb{P}_{\ydiagram{1,1}} \oplus \mathbb{P}_{\ydiagram{2}}\right)[2] & = & \left[  \begin{array}{cccc}
 0 & 0 & 0 & 0 \\ 
 \sm{\fn{12}-\fn{21}} & 0 & 0 & 0 \\
 0 & \sm{\fn{11}} & \sm{ \fn{12} + \fn{21}} & \sm{\fn{22}}
\end{array}\right] \\ \nn
\left(\mathbb{P}_{\ydiagram{1,1,1}} \oplus \mathbb{P}_{\ydiagram{1,1}} \oplus \mathbb{P}_{\ydiagram{2}}\right)[3]\varepsilon & = & \left[
\begin{array}{cc}
 \sm{\fn{123}-\fn{132}-\fn{213}+\fn{231}+\fn{312}-\fn{321}} & 0 \\
 0 & \sm{\fn{12}-\fn{13}-\fn{21}+\fn{23}+\fn{31}-\fn{32}} \\
 0 & 0 \\
\end{array}
\right] \nn \\
\left(\mathbb{P}_{\ydiagram{1,1,1}} \oplus \mathbb{P}_{\ydiagram{1,1}} \oplus \mathbb{P}_{\ydiagram{2}}\right)[4]\varepsilon & = & \left[ \hspace{-15pt}
\begin{array}{c}
  \begin{array}{c}
 \sm{  \; \fn{123}-\fn{124}-\fn{132}+\fn{134}+\fn{142}-\fn{143}-\fn{213}+\fn{214}+\fn{231}-\fn{234}-\fn{241}+\fn{243}} \vspace{-5pt} \\
 \sm{  \hspace{25pt} +\fn{312}-\fn{314}-\fn{321}+\fn{324}+\fn{341}-\fn{342}-\fn{412}+\fn{413}+\fn{421}-\fn{423}-\fn{431}+\fn{432} \;  } 
\end{array} \\
 0  \\
 0
\end{array} \hspace{-5pt}
\right]. \nn
\eea
The map to $\mathbb{P}_{\ydiagram{1,1}} \oplus \mathbb{P}_{\ydiagram{2}}$ is premultiplication by the matrix
\be \nn
\left[ \hspace{-15pt}
\begin{array}{ccc}
 \begin{array}{c}
 \sm{\frac{1}{6} \left( \; \fn{12}-\fn{13}-\fn{21} \right.} \vspace{-5pt} \\
 \sm{ \left. \hspace{25pt} +\fn{23}+\fn{31}-\fn{32} \; \right) }
\end{array} & 0 & 0\\
 0 & \fn{22}-\fn{11} & \;\;\;-\fn{11}+\fn{12}+\fn{21}-\fn{22} \\
\end{array}
\right].
\ee
The basis vectors map to
\bea \nn
&& \left[ \hspace{-10pt}
\phantom{\begin{array}{c}
 0 \\
 0
\end{array}}
\right]
 \nn \\
&& \left[ 
\begin{array}{c}
 0 \\
 0
\end{array}
\right]
 \nn \\
&&\left[
\begin{array}{cccc}
 0 & 0 & 0 & 0 \\
 \sm{2 \left( \fn{22}- \fn{11} \right)} & 0 & \sm{2(-\fn{11}+\fn{12}+\fn{21}-\fn{22})} & 0 \\
\end{array}
\right]
\nn \\
&& \left[ \hspace{-10pt}
\begin{array}{cc}
\begin{array}{c}
 \sm{  \fn{12}-\fn{13}-\fn{21} } \vspace{-5pt} \\
 \sm{ \; \; +\fn{23}+\fn{31}-\fn{32} \; }
\end{array} & 0 \\
 0 & 0 \\
\end{array}
\right]
\nn \\
&&\left[ 
\begin{array}{c}
 0 \\
 0
\end{array}
\right]. \nn
\eea
Typically, the next step is to append identity matrices to these matrices; then, column reduction computes the kernel.  In this case, however, the non-zero columns are linearly independent by inspection.  Thus,
\be \nn
K= \raisebox{-15pt}{\scalebox{2}[5]{$\langle$}} \hspace{-20pt}
\begin{array}{ccccc}
  \begin{array}{c}
 \sm{  \; \fn{123}-\fn{124}-\fn{132}+\fn{134}+\fn{142}-\fn{143}-\fn{213}+\fn{214}+\fn{231}-\fn{234}-\fn{241}+\fn{243}} \vspace{-5pt} \\
 \sm{  \hspace{25pt} +\fn{312}-\fn{314}-\fn{321}+\fn{324}+\fn{341}-\fn{342}-\fn{412}+\fn{413}+\fn{421}-\fn{423}-\fn{431}+\fn{432} \;  } 
\end{array} & 0 & 0 & 0 & 0 \\
 0 & \begin{array}{c}
 \sm{\fn{12}-\fn{13}-\fn{21} } \vspace{-5pt} \\
 \sm{ \hspace{15pt} +\fn{23}+\fn{31}-\fn{32} \; }
\end{array} & 0 & 0 & 0 \\
 0 & 0 & \sm{\fn{11}} & \sm{\fn{22}} & \sm{\fn{11}}
\end{array} \hspace{0pt} \raisebox{-15pt}{\scalebox{2}[5]{$\rangle$}}
\ee
as an imrep.  The block structure of this matrix gives a direct sum decomposition
\be \nn
K \simeq D_4 \oplus D_3 \oplus \tensorpower^1.
\ee
The resulting resolution of $V$ is 
\be \nn
0 \longrightarrow  D_4 \oplus D_3 \oplus \mathbb{P}_{\ydiagram{1}} \longrightarrow \mathbb{P}_{\ydiagram{1,1,1}} \oplus \mathbb{P}_{\ydiagram{1,1}} \oplus \mathbb{P}_{\ydiagram{2}} \longrightarrow \mathbb{P}_{\ydiagram{1,1}} \oplus \mathbb{P}_{\ydiagram{2}} \longrightarrow V \longrightarrow 0,
\ee
From which the dimension formula $\dim V[n] = n$ can be derived.  Koszul complexes expand this out to a resolution by Schur projectives:
\be \nn
\begin{xy}
(15,0)*+{0}="final zero"; (0,0)*+{V}="space";
(-20,0)*+{\mathbb{P}_{\ydiagram{1,1}} \oplus \mathbb{P}_{\ydiagram{2}}}="given cover";
(-67,0)*+{\mathbb{P}_{\ydiagram{1,1}} \oplus \mathbb{P}_{\ydiagram{2}}}="step 1 deg 2";
(-67,20)*+{\mathbb{P}_{\ydiagram{1,1,1}}}="step 1 deg 3";
(-113,-20)*+{\mathbb{P}_{\ydiagram{1}}}="step 2 deg 1";
(-113,20)*+{\mathbb{P}_{\ydiagram{1,1,1}}}="step 2 deg 3";
(-113,40)*+{\mathbb{P}_{\ydiagram{1,1,1,1}}}="step 2 deg 4";
(-130,27)*+{\raisebox{6pt}{\rotatebox{10}{$\ddots$}}}="step 3 deg 4";
(-130,47)*+{\raisebox{6pt}{\rotatebox{10}{$\ddots$}}}="step 3 deg 5";
{\ar "space"; "final zero"};
{\ar "given cover"; "space"};
{\ar_{\sm{\left[
\begin{array}{cc}
 0 & \sm{\fn{22}-\fn{11}} \\
 0 & \sm{-\fn{11}+\fn{12}+\fn{21}-\fn{22}} \\
\end{array}
\right]}} "step 1 deg 2"; "given cover"};
{\ar^{[\partial \; \; \; 0]} "step 1 deg 3"; "given cover"};
{\ar^{[0 \; \; \; \sm{\fn{11}} \,]} "step 2 deg 1"; "step 1 deg 2"};
{\ar^{[\partial \; \; \; 0]} "step 2 deg 3"; "step 1 deg 2"};
{\ar^{[\partial]} "step 2 deg 4"; "step 1 deg 3"};
{\ar "step 3 deg 4"; "step 2 deg 3"};
{\ar "step 3 deg 5"; "step 2 deg 4"};
\end{xy}
\ee
\bibliographystyle{alphaurl}
\bibliography{references}
\end{document}